\documentclass[10pt,a4paper,twoside]{article}
\def\YEAR{\year}\newcount\VOL\VOL=\YEAR\advance\VOL by-1995
\def\firstpage{1}\def\lastpage{1000}
\def\received{}\def\revised{}
\def\communicated{}

\makeatletter
\def\magnification{\afterassignment\m@g\count@}
\def\m@g{\mag=\count@\hsize6.5truein\vsize8.9truein\dimen\footins8truein}
\makeatother

\font\eightrm=cmr8
\font\caps=cmcsc10                    
\font\Caps=cmcsc10 scaled \magstep1   
\font\scaps=cmcsc8

\makeatletter
\@specialpagefalse\headheight=8.5pt
\def\DocMath{{\def\th{\thinspace}\scaps Documenta Math.}}
\renewcommand{\@oddfoot}{\hfill\scaps Documenta Mathematica
    \number\VOL\  (\number\YEAR) \number\firstpage--\lastpage\hfill}
\renewcommand{\@evenfoot}{\ifnum\thepage>\lastpage\hfill\scaps
    Documenta Mathematica \number\VOL\  (\number\YEAR)\hfill\else\@oddfoot\fi}%

\renewcommand{\@evenhead}{%
    \ifnum\thepage>\lastpage\rlap{\thepage}\hfill%
    \else\rlap{\thepage}\slshape\leftmark\hfill{\caps\SAuthor}\hfill\fi}%
\renewcommand{\@oddhead}{%
    \ifnum\thepage=\firstpage{\DocMath\hfill\llap{\thepage}}%
    \else{\slshape\rightmark}\hfill{\caps\STitle}\hfill\llap{\thepage}\fi}%
\makeatother

\def\TSkip{\bigskip}
\newbox\TheTitle{\obeylines\gdef\GetTitle #1
\ShortTitle  #2
\SubTitle    #3
\Author      #4
\ShortAuthor #5
\EndTitle
{\setbox\TheTitle=\vbox{\baselineskip=20pt\let\par=\cr\obeylines%
\halign{\centerline{\Caps##}\cr\noalign{\medskip}\cr#1\cr}}%
	\copy\TheTitle\TSkip\TSkip%
\def\next{#2}\ifx\next\empty\gdef\STitle{#1}\else\gdef\STitle{#2}\fi%
\def\next{#3}\ifx\next\empty%
    \else\setbox\TheTitle=\vbox{\baselineskip=20pt\let\par=\cr\obeylines%
    \halign{\centerline{\caps##} #3\cr}}\copy\TheTitle\TSkip\TSkip\fi%
\centerline{\caps #4}\TSkip\TSkip%
\def\next{#5}\ifx\next\empty\gdef\SAuthor{#4}\else\gdef\SAuthor{#5}\fi%
\ifx\received\empty\relax
    \else\centerline{\eightrm Received: \received}\fi%
\ifx\revised\empty\TSkip%
    \else\centerline{\eightrm Revised: \revised}\TSkip\fi%
\ifx\communicated\empty\relax
    \else\centerline{\eightrm Communicated by \communicated}\fi\TSkip\TSkip%
\catcode'015=5}}\def\Title{\obeylines\GetTitle}
\def\Abstract{\begingroup\narrower
    \parskip=\medskipamount\parindent=0pt{\caps Abstract. }}
\def\EndAbstract{\par\endgroup\TSkip}

\long\def\MSC#1\EndMSC{\def\arg{#1}\ifx\arg\empty\relax\else
     {\par\narrower\noindent%
     2010 Mathematics Subject Classification: #1\par}\fi}

\long\def\KEY#1\EndKEY{\def\arg{#1}\ifx\arg\empty\relax\else
	{\par\narrower\noindent Keywords and Phrases: #1\par}\fi\TSkip}

\newbox\TheAdd\def\Addresses{\vfill\copy\TheAdd\vfill
    \ifodd\number\lastpage\vfill\eject\phantom{.}\vfill\eject\fi}
{\obeylines\gdef\GetAddress #1
\Address #2
\Address #3
\Address #4
\EndAddress
{\def\xs{4.3truecm}\parindent=0pt
\setbox0=\vtop{{\obeylines\hsize=\xs#1\par}}\def\next{#2}
\ifx\next\empty 
     \setbox\TheAdd=\hbox to\hsize{\hfill\copy0\hfill}
\else\setbox1=\vtop{{\obeylines\hsize=\xs#2\par}}\def\next{#3}
\ifx\next\empty 
     \setbox\TheAdd=\hbox to\hsize{\hfill\copy0\hfill\copy1\hfill}
\else\setbox2=\vtop{{\obeylines\hsize=\xs#3\par}}\def\next{#4}
\ifx\next\empty\ 
     \setbox\TheAdd=\vtop{\hbox to\hsize{\hfill\copy0\hfill\copy1\hfill}
                \vskip20pt\hbox to\hsize{\hfill\copy2\hfill}}
\else\setbox3=\vtop{{\obeylines\hsize=\xs#4\par}}
     \setbox\TheAdd=\vtop{\hbox to\hsize{\hfill\copy0\hfill\copy1\hfill}
	        \vskip20pt\hbox to\hsize{\hfill\copy2\hfill\copy3\hfill}}
\fi\fi\fi\catcode'015=5}}\gdef\Address{\obeylines\GetAddress}

\def\LOCAL{\jobname.files}

\makeindex

\usepackage{amsmath, amssymb, amsthm}

\usepackage[all]{xy}

\usepackage{mathtools}
\usepackage{hyperref}

\newcommand{\Hom}{{\rm Hom}}

\newcommand{\Tr}{{\rm Tr}}

\newcommand{\cO}{ {\mathcal{O}}}

\newcommand{\fm}{{\mathfrak{m}}}
\newcommand{\fp}{{\mathfrak{p}}}
\newcommand{\fq}{{\mathfrak{q}}}

\renewcommand{\AA}{{\mathbb A}}
\newcommand{\CC}{{\mathbb C}}
\newcommand{\FF}{{\mathbb{F}}}
\newcommand{\GG}{{\mathbb G}}

\newcommand{\NN}{{\mathbb N}}
\newcommand{\PP}{{\mathbb P}}
\newcommand{\QQ}{{\mathbb Q}}
\newcommand{\RR}{{\mathbb R}}
\newcommand{\ZZ}{{\mathbb Z}}

\newcommand{\be}{\begin{equation}}
\newcommand{\ee}{\end{equation}}

\newtheorem{thm}{Theorem}[section]
\newtheorem*{thm*}{Theorem}
\newtheorem{prop}[thm]{Proposition}
\newtheorem{cor}[thm]{Corollary}
\newtheorem{lem}[thm]{Lemma}
\newtheorem{rmk}[thm]{Remark}
\newtheorem{problem}[thm]{Problem}
\newtheorem{definition}[thm]{Definition}

\begin{document}

\Title
Galois-Module Theory for Wildly Ramified\\
Covers of Curves over Finite Fields\\
\ShortTitle
Galois-Module Theory
\SubTitle

\Author
Helena Fischbacher-Weitz and Bernhard K\"ock
\ShortAuthor
\EndTitle
\vspace*{-5em}
\begin{center}
With an Appendix by \\ \vspace*{0.3em}
{\sc Bernhard K\"ock and Adriano Marmora}
\end{center}
\vspace*{5em}
\Abstract
Given a Galois cover of curves over $\FF_p$, we relate the $p$-adic valuation of epsilon constants appearing in functional equations of Artin L-functions to an equivariant Euler characteristic. Our main theorem generalises a result of Chinburg from the tamely to the weakly ramified case. We furthermore apply Chinburg's result to obtain a `weak' relation in the general case. In the Appendix, we study, in this arbitrarily wildly ramified case, the integrality of $p$-adic valuations of epsilon constants.
\EndAbstract
\MSC
11R58 (Primary) 14G10, 14G15, 11R33, 14H30 (Secondary)
\EndMSC
\KEY
Galois cover of curves; weakly ramified; epsilon constant; equivariant Euler characteristic.
\EndKEY
\Address
Helena Fischbacher-Weitz\\ Mathematical Sciences\\ University of Southampton\\ Southampton SO17 1BJ\\ United Kingdom\\ hfw091011@gmail.com
\Address
Bernhard K\"ock\\ Mathematical Sciences\\ University of Southampton\\ Southampton SO17 1BJ\\ United Kingdom\\ b.koeck@soton.ac.uk
\Address
Adriano Marmora
Institut de Recherche
Math\'ematique Avanc\'ee
Universit\'e de Strasbourg
7 rue Ren\'e Descartes
67084 Strasbourg
France
marmora@math.unistra.fr
\Address
\EndAddress

\section*{Introduction}

Relating global invariants to invariants that are created from local data is a fundamental and widely studied topic in number theory. In this paper, we will study an equivariant version of this local-global principle in the context of curves over finite fields. More precisely, our goal is to relate epsilon constants appearing in functional equations of Artin L-functions to an equivariant Euler characteristic of the underlying curve.

The exact setup considered in this paper is as follows. Let $X$ be an irreducible smooth projective curve over $\FF_p$ and let $G$ be a finite subgroup of $\mathrm{Aut}(X/k)$ where $k$ denotes the algebraic closure of $\FF_p$ in the function field of~$X$.

We begin by defining the two invariants $\chi(G, \bar{X}, \cO_{\bar{X}})$ and $E(G,X)$ that we are going to put in relation with each other.

First, the equivariant Euler characteristic of $\bar{X} := X \times_{\FF_p} \bar{\FF}_p$ is the element
\[\chi(G,\bar{X}, \cO_{\bar{X}}) := \left[H^0(\bar{X}, \cO_{\bar{X}})\right] -\left[H^1(\bar{X}, \cO_{\bar{X}})\right]\]
considered as an element of the Grothendieck group $K_0(G, \bar{\FF}_p)$ of all finitely generated modules over the group ring $\bar{\FF}_p[G]$. We recall that, if $k=\FF_p$, then $H^0(\bar{X}, \cO_{\bar{X}})$ is the trivial representation $\bar{\FF}_p$ and $H^1(\bar{X},\cO_{\bar{X}})$ is isomorphic to the dual $H^0(\bar{X}, \Omega_{\bar{X}})^*$ of the so-called {\em canonical representation} of $G$ on the space $H^0(\bar{X},\Omega_{\bar{X}})$ of global holomorphic differentials on $\bar{X}$.

Second, for any finite-dimensional complex representation $V$ of $G$, we consider the epsilon constant $\varepsilon(V)$ that appears in the functional equation of the Artin L-function associated with $G$, $X$ and the dual representation $V^*$ (see~(\ref{FunctionalEquation})). It is known that $\varepsilon(V) \in \bar{\QQ}$. We denote the standard $p$-adic valuation on $\bar{\QQ}_p^\times$ by $v_p$, we fix a field embedding $j_p: \bar{\QQ} \hookrightarrow \bar{\QQ}_p$, we identify the classical ring $K_0(\CC[G])$ of virtual complex representations with $K_0(\bar{\QQ}[G])$ and we let $j_p$ also denote the isomorphism $K_0(\bar{\QQ}[G]) \;\; \tilde{\rightarrow} \;\; K_0(\bar{\QQ}_p[G])$ induced by $j_p$. We then define the natural element $E(G,X) \in K_0(\bar{\QQ}_p[G])_\QQ := K_0(\bar{\QQ}_p[G])\otimes \QQ$ by the equations
\[\langle E(G,X), j_p(V) \rangle = -v_p(j_p(\varepsilon(V))), \quad \textrm{ for } V \textrm{ as above},\]
where $\langle \hspace{1em}, \hspace{1em} \rangle $ denotes the classical perfect (character) pairing on $K_0(\bar{\QQ}_p[G])$. We call $E(G,X)$ the {\em (virtual) epsilon representation associated with the action of~$G$ on $X$}. We point out (see Section~\ref{GeneralFormula}) that the definition of $E(G,X)$ does not depend on~$j_p$ and that it is compatible with restriction to subgroups of~$G$.

In Section~\ref{GeneralFormula} (which from the logical point of view does not depend on the earlier sections), we will prove the following general relation between  the {\em global} invariant $\chi(G,\bar{X},\cO_{\bar{X}})$ and the invariant $E(G,X)$ which is determined by {\em local} data via the Euler product definition of the Artin L-function. Let $d:K_0(\bar{\QQ}_p[G]) \rightarrow K_0(G, \bar{\FF}_p)$ denote the (surjective) decomposition map from classical modular representation theory.

\begin{thm*}[`Weak' Formula]
We have
\[d(E(G,X)) = \chi(G,\bar{X},\cO_{\bar{X}}) \quad \textrm{in } \quad K_0(G,\bar{\FF}_p)_\QQ.\]
\end{thm*}

While we do not assume any condition on the type of ramification of the corresponding projection
\[\pi:X\rightarrow X/G =:Y\]
for this formula, it is only a `weak' formula in the sense that it does not describe $E(G,X)$ itself, but only its image in $K_0(G,\bar{\FF}_p)$, which for instance, if the order of $G$ is a power of $p$, captures only the rank of $E(G,X)$. The weakness of this formula may also be explained by recalling that two $\bar{\FF}_p[G]$-modules whose classes are equal in $K_0(G,\bar{\FF}_p)$ are not necessarily isomorphic but only have the same composition factors.

The main object of this paper is to establish a `strong' formula. For this, we assume that $\pi$ is {\em weakly ramified}, i.e.\ that the second ramification group~$G_{\fp,2}$ vanishes for all $\fp \in X$. We recall (see Theorem~1.2 on p.~4 in \cite{Pi}) that, by the Deuring-Shafarevic formula, this condition is always satisfied if $X$ is ordinary, i.~e.\ that this condition holds in a sense generically. In order to be able to precisely formulate our strong formula, we introduce the following notations.

Let $\bar{X}^\mathrm{w}$ denote the set of all points $P$ in $\bar{X}$ such that $\bar{\pi}$ is wildly ramified at~$P$. We furthermore define the subset~$\bar{Y}^\mathrm{w}:= \bar{\pi}(\bar{X}^\mathrm{w})$ of $\bar{Y}:= Y \times_{\FF_p} \bar{\FF}_p$ and the divisor $\bar{D}^\mathrm{w}:= - \sum_{P\in \bar{X}^\mathrm{w}} [P]$ on $\bar{X}$. By~\cite{Ko}, the equivariant Euler characteristic
$\chi(G,\bar{X}, \cO_{\bar{X}}(\bar{D}^\mathrm{w}))$ of~$\bar{X}$ with values in the invertible $G$-sheaf $\cO_{\bar{X}}(\bar{D}^\mathrm{w})$ is then equal to the image $c(\psi(G,\bar{X}))$ of a unique element $\psi(G,\bar{X})$ in the Grothendieck group $K_0(\bar{\FF}_p[G])$ of all finitely generated {\em projective} $\bar{\FF}_p[G]$-modules under the Cartan homomorphism $c: K_0(\bar{\FF}_p[G]) \rightarrow K_0(G,\bar{\FF}_p)$. Furthermore, for every $Q \in \bar{Y}$, we fix a point~$\tilde{Q}$ in the fibre $\bar{\pi}^{-1}(Q)$, we denote the decomposition group of $\bar{\pi}$ at~$\tilde{Q}$ by $G_{\tilde{Q}}$ and we denote the trivial representation of rank one by ${\bf 1}$. Finally, let $e:K_0(\bar{\FF}_p[G]) \rightarrow K_0(\bar{\QQ}_p[G])$ denote the third homomorphism from the classical $cde$-triangle.

The following relation between $\psi(G,\bar{X})$ and $E(G,X)$ is the main result of Section~\ref{MainTheoremSection} and of this paper.

\begin{thm*}[`Strong' formula] If $\pi$ is weakly ramified, we have
\begin{equation}\label{StrongFormula}
E(G,X) = e(\psi(G,\bar{X})) + \sum_{Q \in \bar{Y}^\mathrm{w}} \left[\mathrm{Ind}_{G_{\tilde{Q}}}^G ({\bf 1})\right] \quad \textrm{in} \quad K_0(\bar{\QQ}_p[G])_\QQ.
\end{equation}
In particular, $E(G,X)$ belongs to the integral part $K_0(\bar{\QQ}_p[G])$ of $K_0(\bar{\QQ}_p[G])_\QQ$.
\end{thm*}

While the `strong' formula is a generalisation of the (first) main theorem in Chinburg's seminal paper \cite{Ch} (applied to curves), the `weak' formula is basically a corollary of it. More precisely, using Artin's induction theorem for modular representation theory, one quickly sees (see Section~\ref{GeneralFormula}) that it suffices to prove the `weak' formula only in the case when $G$ is cyclic and $p$ does not divide the order of $G$. That case is even more restrictive than the tamely ramified case considered in \cite{Ch}.

In order to prove the `strong' formula, we follow an  approach different to that used in \cite{Ch}. The idea and some of the steps of this alternative approach for tamely ramified covers of curves have been sketched in Erez's beautiful survey article \cite{Er}, but the preprint [CEPT5] cited there and authored by Chinburg, Erez, Pappas and Taylor seems to have not been published.

We now give an overview of our proof of the `strong' formula (\ref{StrongFormula}). We therefore, for almost the entire rest of this introduction, assume that $\pi$ is at most weakly ramified. As already explained earlier, the left-hand side of (\ref{StrongFormula}) is compatible with restriction to any subgroup of $G$. Using Mackey's double coset formula, we will see that the added induced representations on the right-hand side ensure that also the right-hand side is compatible with restriction, in the obvious sense. We need to show that both sides of (\ref{StrongFormula}) are equal after pairing them with $j_p(\mathit{V})$ as above. As usual, the classical Artin induction theorem implies that it suffices to assume that $\mathit{G}$ is cyclic and $\mathit{V}$ corresponds to a multiplicative character~$\chi$. In that case (and in fact also in the slightly more general case when $G$ is abelian), both sides can be explicitly computed as follows.

Let $r$ denote the degree of $k$ over $\FF_p$ and let $g_{Y_k}$ denote the genus of the geometrically irreducible curve $Y/k$. Furthermore, let $Y^\mathrm{t}$ (respectively $Y^\mathrm{w}$) denote the set of all points $\fq \in Y$ such that $\pi$ is tamely ramified (respectively wildly ramified) at one (and then all) point(s) $\fp \in X$ above $\fq$. For $\fq \in Y^t$, we restrict the character $\chi$ to the inertia subgroup $I_{\tilde{\fq}}$ for some $\tilde{\fq} \in \pi^{-1}(\fq)$ and compose it with the norm residue homomorphism from local class field theory to obtain a multiplicative character $\chi_{k(\fq)}$ on the residue field $k(\fq)$ and to obtain the Gauss sum $\tau(\chi_{k(\fq)})\in \bar{\QQ}$. In Section~\ref{EulerCharacteristics}, we will prove the following explicit formula which essentially computes the right-hand side of~(\ref{StrongFormula}).

\begin{thm*}[Equivariant Euler characteristic formula]
If $\pi$ is at most weakly ramified, we have
\begin{equation}\label{EulerCharacteristicFormula}
\langle e(\psi(G,\bar{X})), j_p (\chi) \rangle = r(1 - g_{Y_k}) - \sum_{\fq \in Y^\mathrm{t}} v_p(j_p(\tau(\chi_{k(\fq)}))) - \sum_{\fq \in Y^\mathrm{w}} \deg(\fq) .
\end{equation}
\end{thm*}

The main ingredient in the proof of this theorem is the explicit description of $\psi(G, \bar{X})$ given in \cite{Ko}. After plugging that explicit description into the left-hand side of (\ref{EulerCharacteristicFormula}), it then takes somewhat lengthy calculations to arrive at the right-hand side of (\ref{EulerCharacteristicFormula}). At the end of these calculations we use a variant of Stickelberger's formula for the valuation of Gauss sums that will be developed in Section~\ref{ClassFieldTheory} using local class field theory, particularly the explicit description of the Hilbert symbol.

On the other hand, in Section~\ref{ComputingEpsilon}, we will derive the following formula for the epsilon constant $\varepsilon(\chi)$ from the Deligne-Langlands description of $\varepsilon(\chi)$ as a product of local epsilon constants, see \cite{De}.

\begin{thm*}[Epsilon constant formula]
Let $\pi$ be at most weakly ramified. Then, up to a multiplicative root of unity, we have
\begin{equation}\label{EpsilonConstantFormula}
\varepsilon(\chi^{-1}) = |k| ^{g_{Y_k} -1} \cdot \prod \tau(\chi_{k(\fq)}) \cdot \prod |k(\fq)|
\end{equation}
where the first product runs over all $\fq \in Y$ such that $\chi$ is tamely ramified (but not unramified) at $\fq$ and the second product runs over all $\fq \in Y$ such that $\chi$ is not tamely ramified at $\fq$.
\end{thm*}

To be able to apply the Deligne-Langlands formula we will first recall some preliminaries about the Tamagawa measure on the ring $\AA_{K(Y)}$ of adeles of the function field $K(Y)$ and about parameterising additive characters on $\AA_{K(Y)}/K(Y)$ using differentials. The epsilon constant formula above then follows from the Deligne-Langlands formula after identifying the local epsilon constants in the unramified, tamely and weakly ramified case, respectively.

The Euler characteristic formula (\ref{EulerCharacteristicFormula}) and the epsilon constant formula~(\ref{EpsilonConstantFormula}) finally imply the `strong' formula (\ref{StrongFormula}) after observing that the difference between the set $Y^\mathrm{w}$ and the set of all points $\fq \in Y$ such that $\chi$ is wildly ramified at $\fq$ is accounted for by the sum $\sum_{Q \in \bar{Y}^\mathrm{w}} \left[\mathrm{Ind}_{G_{\tilde{Q}}}^G ({\bf 1})\right]$.

Given our strong and weak formulae, it may be tempting to conjecture that $E(G,X)$ is integral, i.e., belongs to the lattice $K_0(\bar{\QQ}_p[G])$ in $K_0(\bar{\QQ}_p[G])_\QQ$, even if $\pi$ is not weakly ramified. The following theorem (see Theorem~A.2 in the Appendix) indeed provides bounds on the denominators occurring in $E(G,X)$, however we also construct examples in the Appendix showing that these bounds are the best possible in general.

\begin{thm*}
The epsilon representation $E(G,X)$ belongs to $\frac{1}{2} K_0(\bar{\QQ}_p[G])$ if \linebreak[4] $p=2$ and to $\frac{1}{p-1}K_0(\bar{\QQ}_p[G])$ if $p\ge3$.
\end{thm*}

The proof of this theorem relies on the Deligne-Langlands product formula again and on a lemma which computes the $p$-adic valuation of local epsilon constants in the arbitrarily wildly ramified case (see Lemma~A.1).

We finish this introduction by pointing the reader to current related work. For instance, the bibliography of the paper \cite{KT} contains a number of recent papers dealing with the problem of determining the canonical representation of $G$ on $H^0(X,\Omega_X)$ in the wildly ramified case. Current research about Galois module theory for weakly ramified extensions of local or number fields can be found in \cite{BC}, \cite{CV}, \cite{Gr} and \cite{Jo}. Finally, vast generalisations of Deligne's product formula and of some of the material in Section~\ref{ComputingEpsilon} to epsilon constants in $\ell$-adic and $p$-adic cohomology can be found in \cite{Lau} and \cite{Ma}, \cite{AM}, respectively.
\vspace{0.2cm}

{\bf Acknowledgements.} The second author thanks Adriano Marmora for his profound interest in the subject of this paper, for explaining $\ell$- and $p$-adic generalisations of the material in Section~2 and for pointing out a mistake in an earlier generalised version of Lemma~\ref{weakly ramified case}.

\vspace{0.3cm}

{\bf Notations.} Let $p$ be a prime and let $\QQ_p$ denote the $p$-adic completion of the field $\QQ$ of rational numbers. We fix an algebraic closure $\bar{\QQ}_p$ of $\QQ_p$ and denote the residue field of $\bar{\QQ}_p$ by $\bar{\FF}_p$, an algebraic closure of the field $\FF_p$ with $p$ elements. In particular, we have a well-defined reduction map from the ring of integers of $\bar{\QQ}_p$ to $\bar{\FF}_p$ which we denote by $\eta_p$. If $q$ is a power of $p$, the unique subfield of $\bar{\FF}_p$ with $q$ elements is denoted by~$\FF_q$. Furthermore, let $\bar{\QQ}$ denote the algebraic closure of $\QQ$ inside the field~$\CC$ of complex numbers. Throughout this paper, we fix a field embedding $j_p: \bar{\QQ} \rightarrow \bar{\QQ}_p$. The unique extension of the standard $p$-adic valuation on $\QQ$ to $\bar{\QQ}_p$ will be denoted by $v_p$ and takes values in $\QQ$.

For any finite group~$G$ and field~$F$, the Grothendieck group of all finitely generated {\em projective} modules over the group ring $F[G]$ will be denoted by $K_0(F[G])$ and the Grothendieck group of {\em all} finitely generated $F[G]$-modules by $K_0(G,F)$. The isomorphism $K_0(G,\bar{\QQ}) \rightarrow K_0(G,\bar{\QQ}_p)$ induced by the embedding $j_p$ (and other homomorphisms induced by $j_p$) will be denoted by $j_p$ as well. We have a canonical isomorphism between $K_0(G, \bar{\QQ})$ and $K_0(G,\CC)$ and identify these two groups with the classical ring of virtual characters of~$G$. The group of $n^\mathrm{th}$ roots of unity in  $F$ will be denoted by $\mu_n(F)$. If furthermore $H/F$ is a Galois extension, the corresponding trace map is denoted by $\mathrm{Tr}_{H/F}$ or just by $\mathrm{Tr}$.

For any $r \in \RR$, the integral part $\lfloor r \rfloor$ and fractional part $\{r\}$ are related by $\lfloor r \rfloor = r - \{r\}$.

A point in a scheme will always mean a closed point.

\section{A Variant of Stickelberger's Formula}\label{ClassFieldTheory}

Let $L/K$ be a finite {\em abelian} Galois extension of local fields with Galois group~$G$. The multiplicative group $k^\times$ of the residue field of $K$ and the `tame subquotient' $G_0/G_1$ of $G$ are related in two natural ways: on the one hand, the norm residue homomorphism from class field theory induces a natural epimorphism~$\gamma_{L/K}$ from $k^\times$ to $G_0/G_1$; on the other hand, the natural representation of the inertia subgroup $G_0$ of $G$ on the `cotangent space of $L$' induces a monomorphism $\chi_{L/K}$ from $G_0/G_1$ to $k^\times$. In Proposition~\ref{composition} below we compute the endomorphism $\chi_{L/K} \circ \gamma_{L/K}$ of the cyclic group $k^\times$. We use this later in Theorem~\ref{GaussSum} to reformulate Stickelberger's formula for the $p$-adic valuation of the Gauss sum~$\tau(\chi)$ associated with a multiplicative character~$\chi$ of $G_0/G_1$.

We denote the  valuation rings of $L$ and $K$ by $\mathcal{O}_L$ and $\mathcal{O}_K$, the maximal ideals by $\mathfrak{m}_L$ and $\mathfrak{m}_K$ and the residue fields by $l$ and $k$, respectively. The characteristic of $k$ and~$l$ is denoted by~$p >0$ and the cardinality of $k$ is denoted by $q=p^r$. We write $G_i = G_i(L/K)$ and $G^i=G^i(L/K)$ for the $i^\textrm{th}$ higher ramification group of $L/K$ in lower and upper numbering, respectively. Furthermore, let $e^\textrm{t} = e^\textrm{t}_{L/K} = \textrm{ord}(G_0)/\textrm{ord}(G_1)$ be the tame part of the ramification index~$e=e_{L/K}=\textrm{ord}(G_0)$ of~$L/K$, and let $e^\mathrm{w} = e^\mathrm{w}_{L/K}= \textrm{ord}(G_1)$ denote the wild part of $e$. In other words, $e^\mathrm{w}$ is the $p$-part of $e$ and $e^\mathrm{t}$ is the non-$p$-part of $e$.

We re-normalise the norm residue homomorphism $(\hspace{1em}, L/K): K^\times \rightarrow  G$ defined in Chapters~IV and~V of \cite{Ne} by composing it with the homomorphism that takes every element to its inverse and denote the resulting composition by
$$\gamma_{L/K}: K^\times \rightarrow G.$$
By definition, the map $\gamma_{L/K}$ is surjective and its composition with the canonical epimorphism $G \rightarrow G/G_0 \cong \mathrm{Gal}(l/k)$ maps every prime element of $\mathcal{O}_K$ to the inverse of the Frobenius automorphism $x \mapsto x^{q}$. Furthermore, by Theorem~V(6.2) in \cite{Ne}, it maps the group $\mathcal{O}_K^\times$ of units of $\mathcal{O}_K$ onto the subgroup~$G_0$ of $G$ and the subgroup $1+\mathfrak{m}_K$ of $\mathcal{O}_K^\times$ onto the subgroup $G^1$ of $G_1$.  Thus, the norm residue homomorphism $\gamma_{L/K}$ induces the epimorphism
\[\xymatrix{\gamma_{L/K}: k^\times = \mathcal{O}_K^\times/(1+\mathfrak{m}_K) \ar@{->>}[r] & G_0/G^1}\]
(denoted $\gamma_{L/K}$ again).
Hence $G^1 = G_1$, the tame part $e^\mathrm{t}$ of $e$ divides $|k^\times| = q-1$ and the cyclic group $k^\times$ contains all $(e^\mathrm{t})^\mathrm{th}$ roots of unity.

The one-dimensional $l$-representation $\mathfrak{m}_L/\mathfrak{m}_L^2$ of $G_0$ defines a multiplicative character from $G_0$ to $l^\times$ which we may view as a homomorphism
\[\chi_{L/K}: G_0/G_1 \rightarrow k^\times\]
because $G_1$ is a $p$-group and because $k^\times$ contains all $(e^\mathrm{t})^\mathrm{th}$ roots of unity; in fact $\chi_{L/K}$ is injective by Proposition~IV.2.7 in \cite{Se}.

\begin{prop}\label{composition} The composition $\chi_{L/K} \circ \gamma_{L/K}$ raises every element of $k^\times$ to the power $\frac{q-1}{e^\mathrm{t}}\cdot \frac{1}{e^\mathrm{w}}$.
\end{prop}

Here, exponentiating with $\frac{1}{e^\mathrm{w}}$ just means the inverse map of exponentiating with $e^\mathrm{w}$, as usual. This is the identity map if $e^\mathrm{w}$ is a power of $q$ which in turn holds true for instance if $q=p$ or if $e^\mathrm{w} =1$.

\begin{proof}
We first reduce Proposition~\ref{composition} to the case when $L/K$ is tamely ramified. To this end, we consider the diagram
\[\xymatrix{k^\times \ar[rrr]^{\gamma_{L/K} \hspace*{4em}} \ar@{=}[d] &&& G_0(L/K)/G_1(L/K) \ar[rrr]^{\hspace*{4em} \chi_{L/K}} \ar@{->>}[d] &&& k^\times \ar[d]^{e_{L/L^{G_1}}= e^\mathrm{w}}\\
k^\times \ar[rrr]^{\gamma_{L^{G_1}/K}\hspace*{2em}} &&& G_0(L^{G_1}/K) \ar[rrr]^{\hspace*{2em} \chi_{L^{G_1}/K}} &&& k^\times.}
\]
The left-hand square commutes by functoriality of the norm residue homomorphism (see Satz~IV(5.8) in \cite{Ne}). From the definition of $\chi_{L/K}$ and $\chi_{L^{G_1}/K}$ we easily derive that the right-hand square commutes as well. Now, assuming that the lower horizontal composition is the $\frac{q-1}{e^\mathrm{t}}$th power map implies that the upper horizontal composition is the $\left(\frac{q-1}{e^\mathrm{t}} \cdot \frac{1}{e^\mathrm{w}}\right)$th power map, as desired.

Next we claim that we may moreover assume that $L/K$ is totally ramified. Similarly to above, this follows from the commutativity of the diagram
\[\xymatrix{l^\times \ar[rrr]^{\gamma_{L/L^{G_0}} \hspace*{2em}} \ar[d]^N &&& G_0(L/L^{G_0})\ar[rrr]^{\hspace*{2em} \chi_{L/L^{G_0}}} \ar@{=}[d] &&& l^\times\\
k^\times \ar[rrr]^{\gamma_{L/K}\hspace*{2em}} &&& G_0(L/K) \ar[rrr]^{\hspace*{2em} \chi_{L/K}} &&& k^\times\ar@{^{(}->}[u].}
\]
and from the fact that norm map $N: l^\times \rightarrow k^\times$ is the $\frac{q^{[l:k]}-1}{q-1}$th power map and is in particular surjective.

Thus, we may assume that $L/K$ is totally and tamely ramified. In this case, $L/K$ is a cyclic extension and in fact a Kummer extension of the form $L=K(\pi_L)$ where $\pi_L$ is the $e$th root of suitably chosen prime element of $K$. Now, let $a \in k^\times$. Then, by definition of the Hilbert symbol (see Satz~V(3.1) in \cite{Ne}), the element $\chi_{L/K}(\gamma_{L/K}(a))$ is equal to the Hilbert symbol associated with the pair $(a, \pi_L^{-1})$ (the inverse is due to the re-normalisation of the norm residue homomorphism, see above). By the explicit description of the Hilbert symbol given in Satz~V(3.4) of \cite{Ne} this is equal to $a^{\frac{q-1}{e}}$, as was to be shown.
\end{proof}

To explain (the proof of) our variant of Stickelberger's formula it is convenient to introduce the following notation. For any $d \in \ZZ$ we define
\[S_{L/K}(d):=\left\{\left\{\frac{dp^i}{e^\mathrm{t}}\right\}: i= 0, \ldots, r-1\right\},\]
where $\{x\}= x - \lfloor x \rfloor$ denotes the fractional part of any $x \in \RR$. Although we have used set brackets in the definition of $S_{L/K}(d)$, we rather consider $S_{L/K}(d)$ as an unordered tuple, i.e., multiple entries of the same rational number are allowed. As $e^\mathrm{t}$ divides $q-1 = p^r - 1$, we have
\[S_{L/K}(dp^N) = S_{L/K}(d) \quad \textrm{for any } N \in \NN.\]

Let now $\bar{\chi}: G_0/G_1 \rightarrow \bar{\FF}_p^\times$ be a multiplicative character. (We will later define the character $\chi: G_0/G_1 \rightarrow \bar{\QQ}^\times$ corresponding to $\bar{\chi}$.) Furthermore we fix a field embedding $\beta: k \hookrightarrow \bar{\FF}_p$ of $k$ into the algebraic closure $\bar{\FF}_p$ of $\FF_p$. As $\chi_{L/K}$ is injective, there exists a unique integer $d({\bar{\chi}}) \in \{0, \ldots, e^\mathrm{t} - 1\}$ such that ${\bar{\chi}}$ is the $d({\bar{\chi}})^\mathrm{th}$ power of the composition
\[\xymatrix{G_0/G_1 \ar@{^(->}[rr]^{\chi_{L/K}} && k^\times \ar@{^(->}[r]^\beta & \bar{\FF}_p^\times}.\]
While $d({\bar{\chi}})$ depends on the embedding $\beta$, the unordered tuple $S_{L/K}(d({\bar{\chi}}))$ does not. Indeed, any other such embedding is equal to $\beta \circ F^N= \beta^{(p^N)}$ for some $N \in \NN$, where $F: k \rightarrow k$, $x \mapsto x^p$, denotes the Frobenius homomorphisms. Thus, choosing a different embedding amounts to multiplying $d({\bar{\chi}})$ with a power of $p$ and therefore does not change $S_{L/K}(d({\bar{\chi}}))$.

Furthermore, there exists a unique integer $c({\bar{\chi}}) \in \{0, \ldots, q-2\}$ such that the composition
\[\xymatrix{k^\times \ar@{->>}[rr]^{\gamma_{L/K}\hspace*{1.5em}} && G_0/G_1  \ar[r]^{\bar{\chi}} &\bar{\FF}_p^\times}\]
is the $c({\bar{\chi}})^\mathrm{th}$ power of $\beta$.

\begin{cor}\label{EqualityTuples}
We have the following equality of unordered tuples:
\[S_{L/K}(d({\bar{\chi}})) = \left\{\left\{\frac{c({\bar{\chi}})p^i}{q-1} \right\}: i=0, \ldots, r-1\right\}.\]
\end{cor}

\begin{proof}
By definition of $c({\bar{\chi}})$ and of $d({\bar{\chi}})$ and by Proposition~\ref{composition} we have
\begin{align*}
\beta^{c({\bar{\chi}})} = {\bar{\chi}} \circ \gamma_{L/K} = (\beta \circ \chi_{L/K})^{d({\bar{\chi}})} \circ \gamma_{L/K} = \beta^{d({\bar{\chi}}) \cdot \frac{q-1}{e^\mathrm{t}}\cdot \frac{1}{e^\mathrm{w}}}.
\end{align*}
Hence
\[c({\bar{\chi}}) \equiv  \frac{d({\bar{\chi}}) (q-1)}{e^\mathrm{t}} \cdot \frac{q^N}{e^\mathrm{w}} \quad \mathrm{mod} \quad q-1\]
with $N$ chosen big enough so $e^\mathrm{w}$ divides $q^N$. This implies
\[\left\{ \frac{c({\bar{\chi}})}{q-1} \right\} = \left\{  \frac{d({\bar{\chi}})}{e^\mathrm{t}} \cdot \frac{q^N}{e^\mathrm{w}} \right\}\]
and finally
\[S_{L/K}(d({\bar{\chi}})) = S_{L/K} \left(d({\bar{\chi}}) \frac{q^N}{e^\mathrm{w}}\right) = \left\{ \left\{ \frac{c({\bar{\chi}}) p^i}{q-1} \right\}: i=0, \ldots, r-1\right\},\]
as was to be shown.
\end{proof}

We recall that $\eta_p : \mu_{e^t}(\bar{\QQ}_p) \; \stackrel{\sim}{\longrightarrow}\; \mu_{e^t}(\bar{\FF}_p^\times)$ denotes the reduction map modulo $p$ and that we have fixed an emebedding $j_p: \bar{\QQ} \hookrightarrow \bar{\QQ}_p$. There obviously exists a unique character $\chi: G_0/G_1 \rightarrow \bar{\QQ}^\times$ such that $\eta_p j_p \chi = \bar{\chi}$. Composing with the norm residue homomorphism defines the multiplicative character
\[\chi_k := \chi \circ \gamma_{L/K}: k^\times \rightarrow \bar{\QQ}^\times.\]
Let $\zeta_p := e^{\frac{2\pi i}{p}} \in \bar{\QQ} \subset \CC$. We define the additive character
\[ \psi_k: k \rightarrow \bar{\QQ}^\times, \quad x \mapsto \zeta_p^{\mathrm{Tr}(x)} = \exp\left(\frac{2\pi i}{p}{\mathrm{Tr}_{k/\FF_p}(x)}\right),\]
where $\mathrm{Tr}: k \rightarrow \FF_p$ denotes the trace map. Furthermore we define the Gauss sum
\[\tau (\chi):= \sum _{x \in k^\times} \chi_k(x)^{-1} \cdot \psi_k(x) \in \bar{\QQ}.\]
Stickelberger's formula computes the $p$-adic valuation of $\tau(\chi)$. It is usually given in terms of fractions involving the denominator $q-1$, see the proof of Theorem~\ref{GaussSum} below. By definition, our~$\chi_k$ factorises through $G_0/G_1$; this allows us to give the following variant, using fractions involving the denominator $e^\mathrm{t}$.

\begin{thm}\label{GaussSum} We have:
\begin{equation}\label{GaussSumValuation}
v_p\left(j_p \left(\tau(\chi) \right)\right) = \sum_{i=0}^{r-1} \left\{\frac{d(\bar{\chi}) p^i}{e^\mathrm{t}}\right\}.
\end{equation}
\end{thm}

\begin{proof}
By Corollary~\ref{EqualityTuples}, the right-hand side of (\ref{GaussSumValuation}) is equal to
\[\sum_{i=0}^{r-1} \left\{\frac{c({\bar{\chi}})p^i}{q-1} \right\}.\]
Let $s(c(\bar{\chi}))$ denote the sum of digits of the $p$-adic expansion of $c(\bar{\chi})$. Then the previous term is equal to
\[\frac{s(c(\bar{\chi}))}{p-1}\]
by the first two lines of the proof of the Lemma on page~96 in \cite[IV, \S 3]{La}. This in turn is equal to the left-hand side of (\ref{GaussSumValuation}) by Theorem~9 in \cite[IV, \S 3]{La} or by Theorem~27 in \cite{Fr}. (Note that our $\tau(\chi)$ is equal to $\tau(\chi^{-1})$ in \cite{La}, that the distinguished character $\chi_\varphi$ in \cite{La} corresponds to our $\beta^{-1}$ via composing with $\eta_p j_p$ and that the $p$-adic valuation of the number $\omega -1$ in \cite{La} is $(p-1)^{-1}$. Similar remarks apply when Theorem~27 in \cite{Fr} is applied.)
\end{proof}

\section{Computing Epsilon Constants}\label{ComputingEpsilon}

Let $X$ be a smooth projective curve over some finite field~$k$ of characteristic~$p$. We assume that $X$ is geometrically irreducible over ${k}$, i.e.\ that ${k}$ is algebraically closed in the function field $K(X)$ of $X$. Furthermore, let $G$ be a finite subgroup of $\mathrm{Aut}(X/k)$ of order $n$. The goal of this section is to explicitly describe the epsilon constant $\varepsilon(V)$ up to a multiplicative root of unity when $G$ is abelian and the representation~$V$ of $G$ is given by a multiplicative character $\chi: G \rightarrow \bar{\QQ}^\times$, see Theorem~\ref{global epsilon} and Corollary~\ref{EpsilonConstant}. The main ingredient here is Deligne's product formula for epsilon constants as developed in~\cite{De}.

We first recall the notion of epsilon constants. Let
\[\pi: X \rightarrow X/G =: Y\]
denote the canonical projection. The Grothendieck L-function associated with $X$, $G$ and a finite-dimen\-sional $\bar{\QQ}$-representation~$V$ of~$G$ is
\[L(V,t) := \prod_{\fq \in Y} \det\left(1-\mathrm{Frob}(\tilde{\fq})\, t^{\deg(\fq)} \vert V^{I_{\tilde{\fq}}}\right)^{-1};\]
here, $\mathrm{deg}(\fq) := [k(\fq): k]$ denotes the degree of $\fq$, $\tilde{\fq} \in X$ is a (fixed) pre-image of $\fq \in Y$, $I_{\tilde{\fq}}$ denotes the inertia subgroup of $G$ at $\tilde{\fq}$ and $\mathrm{Frob}(\tilde{\fq}) \in G$ denotes a geometric Frobenius automorphism at $\tilde{\fq}$, i.e.\ $\mathrm{Frob}(\tilde{\fq})$ induces the inverse of the usual Frobenius automorphism of the field extension $k(\tilde{\fq})/k(\fq)$. We recall that replacing the geometric with the arithmetic Frobenius automorphism in this definition amounts to defining the Artin L-function which in turn is equal to the Grothendieck L-function $L(V^*,t)$ corresponding to the contragredient representation $V^*$ of $V$, see \cite[Exercise~V.2.21(h)]{Mi}. The Grothendieck L-function satisfies the functional equation
\begin{equation}\label{FunctionalEquation}
L(V,t) = \varepsilon(V) \; t^a \; L(V^*, \frac{1}{|k|t})
\end{equation}
with some $a \in \NN$ and with some $\varepsilon(V) \in \bar{\QQ}$ by \cite[Theorem~VI.13.3]{Mi}. The number $\varepsilon(V)$ is called the epsilon constant associated with $V$ (and $X$, $G$).

Let $K:= K(Y)$ denote the function field of $Y$ and let $\AA_K$ denote the ring of adeles of $K$. We start by explicitly describing additive characters $\psi: \AA_K \rightarrow \bar{\QQ}^\times$ and by giving the Tamagawa measure on $\AA_K$.

For any $\mathfrak{q} \in Y$ let $\hat{\cO}_{Y,\mathfrak{q}}$ denote the completion of the local ring $\cO_{Y,\mathfrak{q}}$ at~$\mathfrak{q}$, let $k(\fq)$ denote its residue field and let $K_\mathfrak{q}$ denote the field of fractions of~$\hat{\cO}_{Y,\fq}$, i.e.\ $K_\mathfrak{q}$ is the completion of the function field $K:=K(Y)$ of $Y$ at $\mathfrak{q}$. The ring $\AA_K$ of adeles of $K$ is then the restricted product, over all $\mathfrak{q} \in Y$, of the fields $K_\mathfrak{q}$ relative to the subrings $\hat{\cO}_{Y,\mathfrak{q}}$. We embed $K$ into $\AA_K$ diagonally and endow $\AA_K$ with its usual topology.

In the next few paragraphs we construct a non-trivial continuous additive character $\psi: \AA_K \to \bar{\QQ}^\times$ that is trivial on $K$. By Proposition~7-15 on p.~270 in \cite{RV}, any other such character~$\tilde{\psi}$ is then given by $\tilde{\psi}(x) = \psi(ax)$ for some unique $a \in K$. While there is a so-called standard character if $K$ is a number field, see Exercise~4 on p.~299 in \cite{RV}, it seems not to be possible to single out a standard character if $K$ is a function field; in particular, the construction outlined in Exercises~5 and~6 on pp.~299-300 in \cite{RV} seems to be flawed. There is however a natural and standard way to parameterize all characters~$\psi$ as above with meromorphic differentials.

Let $\Omega_K$ denote the module of differentials of $K$ over ${k}$, a vector space of dimension~1 over~$K$. For each $\mathfrak{q} \in Y$ let $\Omega_{K_\mathfrak{q}}$ denote $K_\mathfrak{q} \otimes_K \Omega_K$, a vector space over $K_\mathfrak{q}$ of dimension~1.
For each $\mathfrak{q} \in Y$ let
\[\textup{res}_\mathfrak{q}: \Omega_{K_\mathfrak{q}} \to {k}\]
denote the residue map at $\mathfrak{q}$ defined e.g.\ in \cite{Ta2}, see also Theorem~7.14.1 on p.~247 in \cite{Ha}. It can be computed as follows. Let $\pi_\mathfrak{q} \in \cO_{Y,\mathfrak{q}}$ be a local parameter and let $x\,d\pi_\mathfrak{q} \in \Omega_{K_\mathfrak{q}}$. We write $x= \sum_{k = - \infty}^{\infty} \tilde{a}_k \pi_\mathfrak{q}^k$ with `digits' $\tilde{a}_k \in \cO_{Y,\fq}$ representing $a_k \in k(\mathfrak{q})$. Then we have
\[\textup{res}_\mathfrak{q}(x\, d\pi_\mathfrak{q}) = \Tr_{k(\mathfrak{q})/{k}}(a_{-1}) \quad \textup{ in } \quad {k}. \]
For each meromorphic differential $\omega \in \Omega_K$ we now define the additive character
\[\psi_\omega: \AA_K \to \bar{\QQ}^\times, \quad (x_\mathfrak{q})_{\mathfrak{q} \in Y} \mapsto \exp\left(\frac{2\pi i}{p} \Tr_{{k}/\FF_p}\left(\sum_{\mathfrak{q} \in Y} \textup{res}_\mathfrak{q}(x_\mathfrak{q}\omega)\right)\right).\]
Note that the sum on the right-hand side is finite because $x_\mathfrak{q} \in \hat{\cO}_{Y,\mathfrak{q}}$ for almost all $\mathfrak{q}$ and because $\omega$ has at most finitely many poles.

\begin{prop}\label{residue theorem} For each $\omega \in \Omega_K$ the additive character $\psi_\omega$ is trivial on~$K$.
\end{prop}

\begin{proof}
This follows from the residue theorem, see Corollary on p.~155 in \cite{Ta2} or Theorem 7.14.2 on p.~248 in \cite{Ha}.
\end{proof}

The quotient $\AA_K/K$ is compact by Theorem~5-11 on p.~192 in \cite{RV}. The following proposition computes its volume. Let $g_Y$ denote the genus of $Y$.

\begin{prop}\label{Tamagawa measure}
For each $\mathfrak{q} \in Y$ let $dx_\mathfrak{q}$ be that Haar measure on the additive group $K_\mathfrak{q}$ for which the volume of  $\hat{\cO}_{Y,\mathfrak{q}}$ is equal to $1$. Then the volume of $\AA_K/K$ with respect to the measure $\prod_{\mathfrak{q} \in Y} {\rm d}x_\mathfrak{q}$ on~$\AA_K$ is equal to~$|k|^{g_Y-1}$.
\end{prop}

\begin{proof}
This is Corollary~1 in Chapter~VI on p.~100 in Weil's book \cite{We}. The first arXiv version of this article contains a modern proof using cohomological methods.
\end{proof}

Recall that $\pi:X \to Y$ is a non-constant finite morphism between geometrically irreducible smooth projective curves over ${k}$ such that the corresponding extension $K(X)/K$ of function fields is a Galois extension with Galois group~$G$ of order $n$. Henceforth, we assume that $G$ is abelian, we fix a multiplicative character
\[\chi: G \to \bar{\QQ}^\times\]
and we denote the corresponding epsilon constant by $\varepsilon(\chi)$.

As above, for each $\fq \in Y$, let ${\rm d} x_\fq$ be that Haar measure on the additive group~$K_\fq$ for which the volume of $\hat{\cO}_{Y,\fq}$ is equal to $1$. Furthermore we define an additive character $\psi_\fq$ on $K_\fq$ that is trivial on $\hat{\cO}_{Y,\fq}$ but not anymore on $\fm_\fq^{-1}$ where $\fm_\fq$ denotes the maximal ideal of $\hat{\cO}_{Y,\fq}$. To this end we fix a generator $\pi_\fq \in \cO_{Y,\fq}$ of the ideal $\fm_\fq$ and, given any $x \in K_\fq$, we write $x= \sum_{k=-\infty}^\infty \tilde{a}_k \pi_\fq^k$ with `digits' $\tilde{a}_k \in \hat{\cO}_{Y,\fq}$ representing $a_k \in k(\fq)$; we then define
\[\psi_\fq: K_\fq \to \bar{\QQ}^\times, \quad x \mapsto \psi_\fq(x):= \exp\left(\frac{2\pi i}{p} \Tr_{k(\fq)/\FF_p}(a_{-1})\right).\]
The restriction of the character $\chi$ to the decomposition group $G_{\tilde{\fq}}$ of some $\tilde{q} \in X$ lying above $\fq$ is denoted by $\chi_{\tilde{\fq}}$. Let $\varepsilon(\chi_{\tilde{\fq}}, \psi_{\fq}, {\rm d} x_\fq)$ denote the local epsilon constant associated with $\chi_{\tilde{\fq}}$, $\psi_\fq$ and ${\rm d} x_\fq$, as defined in \S 4 of~\cite{De}.

For any two complex numbers $w,z$ we write $w \sim z$ if there exists a root of unity $\zeta$ (i.e.\ $\zeta \in \exp(2\pi i \QQ)$) such that $w = \zeta z$. Note that this equivalence relation is finer than the equivalence relation $\sim$ defined in the Appendix of~$\cite{De}$.

\begin{thm}\label{global epsilon}
We have
\[\varepsilon(\chi) \sim |k|^{g_Y-1} \prod_{\mathfrak{q} \in Y} \varepsilon(\chi_{\tilde{\fq}}, \psi_{\fq}, {\rm d} x_\fq).\]
\end{thm}

\begin{proof}[Proof (see also Remarque~(3.1.3.6) in \cite{Lau} for parts of this proof)]
We fix a non-zero meromorphic differential $\omega \in \Omega_K$. By Proposition~\ref{residue theorem}, the differential $\omega$ determines a non-trivial continuous additive character
\[\psi_\omega: \AA_K \to \bar{\QQ}^\times\]
that is trivial on $K$. By Proposition~\ref{Tamagawa measure}, the measure $q ^{1-g_Y} \cdot \prod_{\fq \in Y} {\rm d} x_\fq$ is then a Tamagawa measure on $\AA_K$, i.e.\ the volume of $\AA_K/K$ is equal to $1$. By (5.11.2), (5.3) and Th\'eor\`eme~7.11(iii) in \cite{De}, the epsilon constant~$\varepsilon(\chi)$ can be decomposed as a product of local epsilon constants as follows:
\[\varepsilon(\chi) = |k|^{1-g_Y} \prod_{\fq \in Y} \varepsilon(\chi_{\tilde{\fq}}, \psi_{\omega,\fq}, {\rm d} x_\fq);\]
here the local additive character $\psi_{\omega,\fq}$ is given by
\[\psi_{\omega,\fq}: K_\fq \to \bar{\QQ}^\times, \quad x \mapsto \exp\left(\frac{2\pi i}{p} \Tr_{{k}/\FF_p}(\textup{res}_\fq(x\omega))\right).\]
We now write $\omega = y_\fq {\rm d} \pi_\fq$ with some $y_\fq \in K^\times \subset K_\fq^\times$. Then we obviously have
\[\psi_{\fq}(x) = \psi_{\omega,\fq}(y_{\fq}^{-1}x) \quad \text{ for } x \in K_\fq.\]
Recall that the surjective norm-residue homomorphism
\[\gamma_{\tilde{\fq}}: K_\fq^\times \twoheadrightarrow G_{\tilde{\fq}}\]
(see Section~\ref{ClassFieldTheory}) maps the group $\hat{\cO}_{Y,\fq}^\times$ of units onto the inertia subgroup $I_{\tilde{\fq}} = G_{\tilde{q},0}$ of $G_{\tilde{\fq}}$.
Let now
\[\xymatrix{ \underline{\chi}_{\tilde{\fq}} :  K_\fq^\times \ar[r]^{\hspace{1em} \gamma_{\tilde{\fq}}} & G_{\tilde{\fq}} \ar[r]^{\chi_{\tilde{\fq}}} & \bar{\QQ}^\times}\]
denote the composition of $\chi_{\tilde{\fq}}$ with the norm-residue homomorphism and let $n_\fq$ denote the valuation of $y_\fq$ at $\fq$. By Formula~(5.4) in \cite{De} we then have
\[\varepsilon(\chi_{\tilde{\fq}}, \psi_{\omega,\fq}, {\rm d}x_\fq) = \underline{\chi}_{\tilde{\fq}} (y_\fq)\; |k(\fq)|^{n_\fq} \; \varepsilon(\chi_{\tilde{\fq}}, \psi_\fq, {\rm d}x_\fq) \sim |k(\fq)|^{n_\fq} \; \varepsilon(\chi_{\tilde{\fq}}, \psi_\fq, {\rm d}x_\fq).\]
In fact we have equality here whenever $\pi$ is unramified at $\tilde{\fq}$ and $\omega$ does not have a pole or zero at $\fq$; in particular we have equality for almost all $\fq \in Y$. Moreover we have
\[\sum_{\fq \in Y} \deg(\fq) n_\fq = {\rm deg}(\omega) = 2g_Y-2\]
by Example~1.3.3 on p.~296 in \cite{Ha}. Thus we obtain
\[\varepsilon(\chi) \sim |k| ^{g_Y-1} \prod_{\fq \in Y} \varepsilon(\chi_{\tilde{\fq}}, \psi_\fq, {\rm d}x_\fq),\]
as was to be shown.
\end{proof}

The following lemmas compute $\varepsilon(\chi_{\tilde{q}},\psi_\fq, {\rm d}x_\fq)$ in three different cases. Recall that the character $\chi$ is said to be {\em unramified}, {\em tamely ramified} or {\em weakly ramified} at $\fq$ if $\chi$ is trivial on the inertia subgroup $I_{\tilde{\fq}}=G_{\tilde{\fq},0}$, the first ramification group $G_{\tilde{\fq},1}$ or the second ramification group $G_{\tilde{\fq},2}$, respectively.

The first lemma concerns the unramified case and in particular tells us that the product in Theorem~\ref{global epsilon} is indeed a finite product.

\begin{lem}\label{unramified case}
If $\chi$ is unramified at $\fq$, then
\[\varepsilon(\chi_{\tilde{\fq}}, \psi_\fq, {\rm d}x_\fq) =1.\]
\end{lem}

\begin{proof}
This is stated in Section~5.9 of \cite{De}.
\end{proof}

We now assume that $\chi$ is tamely ramified at $\fq$. Then the character $\underline{\chi}_{\tilde{\fq}}: K_\fq^\times \to \bar{\QQ}^\times$ (defined in the proof of Theorem~\ref{global epsilon}) is trivial on $1+\fm_\fq$ by Theorem~V(6.2) in \cite{Ne} and hence induces a multiplicative character
\[\chi_{k(\fq)}: k(\fq)^\times \to \bar{\QQ}^\times\]
on the finite field $k(\fq)$. Furthermore, as in Section~\ref{ClassFieldTheory}, we introduce the standard additive character
\[\psi_{k(\fq)}: k(\fq) \to \bar{\QQ}^\times, \quad x \mapsto \exp\left(\frac{2\pi i}{p} \Tr_{k(\fq)/\FF_p}(x)\right).\]
and the Gauss sum
\[\tau(\chi_{k(\fq)}) := \sum_{x \in k(\fq)^\times} \chi_{k(\fq)}^{-1}(x)\psi_{k(\fq)}(x)\]
associated with $\chi_{k(\fq)}$ and $\psi_{k(\fq)}$.

\begin{lem}\label{tamely ramified case}
If $\chi$ is tamely ramified at $\fq$, then
\[\varepsilon(\chi_{\tilde{\fq}}, \psi_\fq, {\rm d}x_\fq) \sim \tau(\chi_{k(\fq)}).\]
\end{lem}

\begin{proof} By the previous proposition we may assume that $\chi$ is not unramified at~$\fq$.
Let the character $\psi'_\fq$ be defined by
\[\psi'_\fq(x) = \psi_\fq(\pi_\fq^{-1} x) \quad \text{for } x \in K_\fq.\]
Furthermore let $\overline{{\rm d}x_\fq}$ denote that Haar measure on $K_\fq$ for which the volume of $\fm_\fq$ is equal to $1$, i.e.\ $\overline{{\rm d}x_\fq} = |k(\fq)|\; {\rm d}x_\fq$. By Formulas~(5.3) and~(5.4) and Section~5.10 in \cite{De} we then have
\[\varepsilon(\chi_{\tilde{\fq}},\psi_\fq, {\rm d}x_\fq) = \underline{\chi}_{\tilde{\fq}}(\pi_\fq) \varepsilon(\chi_{\tilde{\fq}},\psi'_\fq, \overline{{\rm d}x_\fq}) \sim \varepsilon(\chi_{\tilde{\fq}},\psi'_\fq, \overline{{\rm d}x_\fq}) = \tau(\chi_{k(\fq)})\]
because $\psi'_\fq$ obviously induces $\psi_{k(\fq)}$ and the character $\chi_{k(\fq)}$ is non-trivial.
\end{proof}

In the final lemma we treat the weakly ramified case.

\begin{lem}\label{weakly ramified case}
If $\chi$ is weakly but not tamely ramified at $\fq$, then
\[\varepsilon(\chi_{\tilde{\fq}}, \psi_\fq, {\rm d}x_\fq) \sim |k(\fq)|.\]
\end{lem}

This is a special case of Lemma~A.1 in the Appendix. Rather than studying that much more general result the reader may prefer the following fairly quick proof.

\begin{proof}
The isomorphisms
\[\fm_\fq/\fm_\fq^2\;\; \tilde{\to}\;\; (1+\fm_\fq)/(1+\fm_\fq^2), \quad \overline{a} \mapsto \overline{1+a},\]
and
\[\fm_\fq^{-2}/\fm_\fq^{-1} \;\; \tilde{\to}\;\;\ \Hom\left(\fm_\fq/\fm_\fq^2, \bar{\QQ}^\times\right), \quad \overline{\gamma} \mapsto (\overline{a} \mapsto \psi_\fq(\overline{\gamma a})),\]
show that there exists a $\gamma \in \fm_\fq^{-2} \backslash \fm_\fq^{-1}$, unique modulo $\fm_\fq^{-1}$, such that
\[\chi_{\tilde{\fq}}(1+a) = \psi_\fq(\gamma a) \quad \text{for all } a \in \fm_\fq.\]
By Formula~(3.4.3.2) in \cite{De} we have
\[\varepsilon(\chi_{\tilde{\fq}}, \psi_\fq, {\rm d}x_\fq) = \int_{\gamma \hat{\cO}_{Y,\fq}^\times} \chi_{\tilde{\fq}}^{-1}(x) \psi_\fq(x) {\rm d}x_\fq = \sum_{b \in k(\fq)^\times} \int_{\gamma \tilde{b}(1+\fm_\fq)} \chi_{\tilde{\fq}}^{-1}(x) \psi_\fq(x) {\rm d}x_\fq\]
where $\tilde{b} \in \hat{\cO}_{Y,\fq}^\times$ represents $b \in k(\fq)^\times$. For all $b \in k(\fq)^\times$ and $a \in \fm_\fq$ we have
\[\chi_{\tilde{\fq}}^{-1}(\gamma \tilde{b}(1+a)) \psi_\fq(\gamma \tilde{b} (1+a))) =  \chi_{\tilde{\fq}}^{-1}(\gamma \tilde{b}) \psi_\fq(\gamma \tilde{b}) \psi_\fq(\gamma(\tilde{b}-1)a).\]
If $\tilde{b}=1$, the function $\chi_{\tilde{\fq}}^{-1}(x) \psi_\fq(x)$ is therefore constant on $\gamma \tilde{b} (1+\fm_\fq)$ and we hence obtain
\[\begin{aligned}
\MoveEqLeft\int_{\gamma \tilde{b} (1+\fm_\fq)} \chi_{\tilde{\fq}}^{-1}(x) \psi_\fq(x) {\rm d}x_\fq
= \chi_{\tilde{\fq}}^{-1}(\gamma) \psi_\fq(\gamma) \mathrm{vol}(\gamma \tilde{b} + \gamma \fm_\fq)\\
&\hspace*{16em} = \chi_{\tilde{\fq}}^{-1}(\gamma) \psi_\fq(\gamma) |k(\fq)|.
\end{aligned}\]
If $b \not= 1$, the substitution $x = \gamma \tilde{b}(1+a)$ gives ${\rm d}x_\fq = |k(\fq)|^2 {\rm d}a$ and hence
\[\int_{\gamma \tilde{b}(1+\fm_\fq)} \chi_{\tilde{\fq}}^{-1}(x) \psi_\fq(x) {\rm d} x_\fq
= \chi_{\tilde{\fq}}^{-1}(\gamma \tilde{b}) \psi_\fq(\gamma \tilde{b}) |k(\fq)|^2 \int_{\fm_\fq} \psi_\fq(\gamma (\tilde{b} -1) a) {\rm d} a =0\]
because $\psi_\fq$ is a non-trivial additive character on $\fm_\fq^{-1}/\hat{\cO}_{Y,\fq}$. Hence
\[\varepsilon(\chi_{\tilde{\fq}}, \psi_\fq, {\rm d}x_\fq) = \chi_{\tilde{\fq}}^{-1}(\gamma)\psi_\fq(\gamma)|k(\fq)| \sim |k(\fq)|,\]
as was to be shown.
\end{proof}

\begin{cor}\label{EpsilonConstant}
If $\chi$ is weakly ramified at all $\fq \in Y$, we have
\[\varepsilon(\chi) \sim |k| ^{g_Y -1} \cdot \prod \tau(\chi_{k(\fq)}) \cdot \prod |k(\fq)|\]
where the first product runs over all $\fq \in Y$ such that $\chi$ is tamely ramified (but not unramified) at $\fq$ and the second product runs over all $\fq \in Y$ such that $\chi$ is not tamely ramified at $\fq$.
\end{cor}

\begin{proof}
This immediately follows from Theorem~\ref{global epsilon} and Lemmas~\ref{unramified case}, \ref{tamely ramified case} and~\ref{weakly ramified case}.
\end{proof}

\section{Computing Equivariant Euler Characteristics}\label{EulerCharacteristics}

Let $k$ and $\pi: X \rightarrow X/G =: Y$ be as in the previous section; in particular, $G$ is assumed to be abelian. We moreover assume in this section that $\pi$ is at most weakly ramified and we fix $q$ such that $\FF_q \subseteq k$. (In the next sections, $q$ will be equal to $p$, but almost all computations in this section can also be done for $\FF_q = k$ and may be of independent interest.)

From the Euler characteristic of a certain equivariant invertible module on $\bar{X} := X \times_{\FF_q} \bar{\FF}_p$ we obtain  a distinguished virtual representation in $K_0(G, \bar{\QQ}_p)$, see below.
The object of this section to find the multiplicity of each irreducible character of $G$ in this representation; see Theorem~\ref{MainFormula} and Corollary~\ref{q=p} for a precise formulation of our results. To prove this we apply the main theorem of~\cite{Ko} and repeatedly re-interpret the resulting terms in a lengthy `metamorphosis'.

For each $\mathfrak{p} \in X$ let $e_\mathfrak{p}^\mathrm{w}$ and $e_\mathfrak{p}^\mathrm{t}$ denote the wild part (i.e., $p$-part) and tame part (i.e., non-$p$-part) of the ramification index~$e_\mathfrak{p}$ of $\pi$ at $\mathfrak{p}$, respectively. Furthermore, let $G_\mathfrak{p} := \{\sigma \in G: \sigma(\mathfrak{p}) = \mathfrak{p}\}$, $I_\mathfrak{p} := \mathrm{ker}(G_\mathfrak{p} \rightarrow \mathrm{Aut}(k(\mathfrak{p})))$ and $G_{\mathfrak{p},1}$ denote the decomposition group, inertia group and first ramification group at~$\mathfrak{p}$, respectively. We recall that the exponent of $G_{\mathfrak{p},1}$ is $p$ and that $e_\mathfrak{p} = \mathrm{ord}(I_\mathfrak{p})$, $e_\mathfrak{p}^\mathrm{w} = \mathrm{ord}(G_{\mathfrak{p},1})$ and $e_\mathfrak{p}^\mathrm{t} = \mathrm{ord} (I_\mathfrak{p}/G_{\mathfrak{p},1})$. If $e_\mathfrak{p}^\mathrm{w} = 1$, the map $\pi$ is said to be {\em tamely ramified} at $\mathfrak{p}$, otherwise we say it is {\em wildly ramified} at $\mathfrak{p}$.

Our assumptions imply the following somewhat unexpected dichotomy.

\begin{lem}\label{dichotomy}
For each $\mathfrak{p} \in X$ we have $e_\mathfrak{p}^\mathrm{t}=1$ or $e_\mathfrak{p}^\mathrm{w} = 1$.
\end{lem}

\begin{proof}
As $\pi$ is weakly ramified, the conjugation action of $I_\fp/G_{\fp,1}$ on $G_{\fp,1}\backslash \{\mathrm{id}\}$ is free by Proposition~9 in \S 2, Ch.\ IV of \cite{Se}. As $G$ is abelian, this can only happen if $e_\mathfrak{p}^\mathrm{t}=1$ or $e_\mathfrak{p}^\mathrm{w} = 1$.
\end{proof}

Let
\[\bar{\pi}: \bar{X} := X \times_{\FF_q} \bar{\FF}_p \rightarrow Y \times_{\FF_q} \bar{\FF}_p =: \bar{Y}\]
denote the morphism induced by $\pi$.  Let $\bar{X}^\mathrm{w}$ denote the set of $P \in \bar{X}$ such that $\bar{\pi}$ is wildly ramified at~$P$, let $\bar{Y}^\mathrm{w}:= \pi(\bar{X}^\mathrm{w})$ and let the divisor $\bar{D}^\mathrm{w}$ on $\bar{X}$ be defined by
\[\bar{D}^\mathrm{w}:= -\sum_{P \in \bar{X}^\mathrm{w}} [P].\]
For later use, we here also introduce the set $\bar{Y}^\mathrm{t}$ of points $Q\in \bar{Y}$ such that $\bar{\pi}$ is tamely ramified but not unramified above $Q$. The subset $X^\mathrm{w}$ of $X$ and the subsets $Y^\mathrm{w}$ and $Y^\mathrm{t}$ of $Y$ are similarly defined.

By Theorem~2.1(a) in \cite{Ko}, applied to each irreducible component of $\bar{X}$, the equivariant Euler characteristic
\[\chi(G, \bar{X}, \mathcal{O}_{\bar{X}}(\bar{D}^\mathrm{w})) := \left[H^0(\bar{X}, \mathcal{O}_{\bar{X}}(\bar{D}^\mathrm{w}))\right] - \left[H^1(\bar{X}, \mathcal{O}_{\bar{X}}(\bar{D}^\mathrm{w}))\right] \in K_0(G, \bar{\FF}_p)\]
of $\bar{X}$ with values in the $G$-equivariant invertible $\mathcal{O}_{\bar{X}}$-module $\mathcal{O}_{\bar{X}}(\bar{D}^\mathrm{w})$ lies in the image of the (injective) Cartan homomorphism
\[c: K_0(\bar{\FF}_p[G]) \rightarrow K_0(G, \bar{\FF}_p).\]
We define $\psi(G, \bar{X}) \in K_0(\bar{\FF}_p[G])$ by
\[c(\psi(G, \bar{X})) = \chi(G, \bar{X}, \mathcal{O}_{\bar{X}}(\bar{D}^\mathrm{w})).\]

Let
\[\xymatrix{& K_0(G,\bar{\QQ}_p) \ar_d[rd] \\  K_0(\bar{\FF}_p[G]) \ar^c[rr] \ar_e[ru] && K_0(G, \bar{\FF}_p)
}\]
be the so-called $cde$-triangle and let
\[\langle \hspace{1em}, \hspace{1em}  \rangle: K_0(G, \bar{\QQ}_p) \times K_0(G,\bar{\QQ}_p) \rightarrow \ZZ\]
denote the classical character pairing
(given by
\[\langle [M], [N]\rangle = \dim_{\bar{\QQ}_p} \Hom_{\bar{\QQ}_p[G]}(M,N)\]
for any finitely generated $\bar{\QQ}_p[G]$-modules $M$, $N$).

Let $\chi:G \rightarrow \bar{\QQ}^\times$ be a multiplicative character of $G$, which we also consider as a map to the group $\mu_n(\bar{\QQ})$ of $n^\textrm{th}$ roots of unity in $\bar{\QQ}$ and also as an element of~$K_0(G,\bar{\QQ})$. The main theorem of this section, see Theorem~\ref{MainFormula} below, will give a formula for $\langle e(\psi(G,\bar{X})), j_p \chi \rangle$, i.e., the multiplicity of the character $j_p\chi$ in the virtual $\bar{\QQ}_p$-representation $e(\psi(G, \bar{X}))$ of $G$. To state the theorem we need to introduce further notations.

Let $r$ denote the degree of $k$ over $\FF_q$ and let $g_{Y_k}$ denote the genus of the geometrically irreducible curve $Y/k$. For each $\fp \in X$ let
\[\chi_\fp :I_\fp \rightarrow k(\fp)^\times\]
denote the multiplicative character corresponding to the obvious representation of $I_\fp$ on the one-dimensional cotangent space of $X$ at $\fp$. Note that, in contrast to the previous section, the restriction of $\chi$ to the decomposition group $G_\fp$ is not denoted by $\chi_\fp$ but by $\mathrm{Res}^G_{G_\fp}(\chi)$ in this and the next sections.

For each $\fq \in Y$, we let $\mathrm{deg}(\fq) := [k(\fq): \FF_q]$ denote the degree of $\fq$, we fix a point $\tilde{\fq} \in \pi^{-1}(\fq) \subset X$ and a field embedding $\alpha_{\tilde{\fq}}: k(\tilde{\fq}) \hookrightarrow \bar{\FF}_p$ (fixing $\FF_q$) and we write $e_\fq:= e_{\tilde{\fq}}, e_\fq^\mathrm{t} := e_{\tilde{\fq}}^\mathrm{t}$, etc. If $\pi$ is tamely ramified at $\tilde{\fq}$, the character $\chi_{\tilde{\fq}}$ is injective by Proposition~IV.2.7 in \cite{Se} and we define $d_\fq(\chi)$ to be the unique integer $d \in \{0, \ldots, e_\fq - 1\}$ such that the composition
\[\mathrm{Res}^G_{I_{\tilde{\fq}}} (\eta_p j_p \chi): \xymatrix{I_{\tilde{\fq}} \ar@{^(->}[r]& G \ar[r]^\chi & \mu_n(\bar{\QQ}) \ar[r]^{j_p} & \mu_n(\bar{\QQ}_p) \ar[r]^{\eta_p} & \mu_n(\bar{\FF}_p) }\]
is the $d^\textrm{th}$ power of the composition
\[\xymatrix{I_{\tilde{\fq}} \ar[r]^{\chi_{\tilde{\fq}}} & k(\tilde{\fq})^\times \ar[r]^{\alpha_{\tilde{\fq}}} & \bar{\FF}_p.}\]

\begin{thm}\label{MainFormula}
We have
\begin{equation}\label{Formula}
\begin{aligned}
 \langle e(\psi(G, \bar{X})), j_p \chi \rangle
= r(1- g_{Y_k})- \sum_{\mathfrak{q} \in Y^\mathrm{t}} \sum_{i=0}^{\deg(\fq)-1} \left\{\frac{d_\fq(\chi)  
q^i}{e_\mathfrak{q}} \right\} - \sum_{\fq \in Y^\mathrm{w}} \deg(\fq). 
\end{aligned}
\end{equation}
\end{thm}

\begin{proof}
We have a well-defined pairing
\[\langle \hspace{1em}, \hspace{1em} \rangle: K_0(\bar{\FF}_p[G]) \times K_0(G, \bar{\FF}_p) \rightarrow \ZZ\]
given by
\[\langle [P], [M] \rangle = \dim_{\bar{\FF}_p} \Hom_{\bar{\FF}_p[G]}(P, M)\]
for any finitely generated projective $\bar{\FF}_p[G]$-module $P$ and any finitely generated $\bar{\FF}_p[G]$-module $M$. By \cite[15.4b)]{Se2} we have
\[\langle e(x), y \rangle = \langle x, d(y) \rangle\]
for all $x \in K_0(\bar{\FF}_p[G])$ and $y \in K_0(G, \bar{\QQ}_p)$.

Hence the left-hand side of (\ref{Formula}) is equal to
\begin{equation}\label{eins}
\langle \psi(G, \bar{X}), d(j_p \chi) \rangle.
\end{equation}
For any point $P$ of $\bar{X}$ let $I_P$ denote the inertia group of $\bar{\pi}$ at $P$. Recall that $I_P$ is equal to the decomposition group $G_P = \{\sigma \in G: \sigma(P) =P\}$ and that $e_P = \mathrm{ord}(I_P)$ is the ramification index. The multiplicative character $\chi_P: I_P \rightarrow \bar{\FF}_P^\times$ afforded by the representation of $I_P$ on the 1-dimensional cotangent space at $P$ is given by
\[ \chi_P(\sigma) t_P \equiv \sigma(t_P) \quad \mathrm{mod } \quad  (t_P^2)\]
where $\sigma \in I_P$ and $t_P$ is a local parameter at $P$. For any $d \in \ZZ$, the $d^\mathrm{th}$~(tensor) power of $\chi_P$ is denoted by $\chi_P^d$.

By Theorems~4.3 and 4.5 in \cite{Ko}, applied to each of the $r$ irreducible components of $\bar{X}$, we have
\begin{equation}
\begin{aligned}\label{GaloisStructureTheorem}
\MoveEqLeft{\psi(G, \bar{X})}\\
= &- \frac{1}{n} \sum_{P \in \bar{X}} \sum_{d=1}^{e_P^\mathrm{t} -1} e_P^\mathrm{w} \, d \left[\mathrm{Ind}_{I_P}^G\left(\mathrm{Cov}(\chi_P^d)\right)\right]
 + \left(r(1-g_{Y_k}) - |\bar{Y}^\mathrm{w}| \right) [\bar{\FF}_p[G]]
\end{aligned}
\end{equation}
in $K_0(\bar{\FF}_p[G])$; here,
``Cov'' means ``projective cover of''. (Note that the integers $l_P$ and $m_P$ occurring in Theorem~4.5 of \cite{Ko} are as follows for the divisor $\bar{D}^\mathrm{w}$: $l_P=0$ for all $P \in \bar{X}$, $m_P=-1$ for all $P \in \bar{X}^\mathrm{w}$ and $m_P=0$ else. Also note that, despite the fraction $\frac{1}{n}$, the first of the two summands above is an element of $K_0(\bar{\FF}_p[G])$, see Theorem~4.3 in~\cite{Ko}.)

By definition of the decomposition map $d$, the element $d(j_p \chi)$ is equal to the class of the 1-dimensional $\bar{\FF}_p$-representation of $G$ given by the composition
\[\xymatrix{G \ar[r]^\chi & \mu_n(\bar{\QQ}) \ar[r]^{j_p} & \mu_n(\bar{\QQ}_p) \ar[r]^{\eta_p} &\mu_n(\bar{\FF}_p).}\]
Hence (\ref{eins}) is equal to
\begin{equation}
\begin{aligned}\label{zwei}
&{- \frac{1}{n} \sum_{P \in \bar{X}} \sum_{d=1}^{e_P^\mathrm{t} -1} e_P^\mathrm{w} \, d \left\langle \mathrm{Ind}_{I_P}^G \left(\mathrm{Cov}(\chi_P^d)\right), \eta_p j_p \chi \right\rangle}\\
&+ \left(r(1-g_{Y_k}) - |\bar{Y}^\mathrm{w}| \right) \left\langle \bar{\FF}_p[G], \eta_p j_p \chi \right\rangle.
\end{aligned}
\end{equation}
If $H$ is a subgroup of $G$, $P$ a finitely generated projective $\bar{\FF}_p[H]$-module and $M$ a finitely generated $\bar{\FF}_p[G]$-module, then we have
\[\begin{aligned}
\MoveEqLeft {\langle \mathrm{Ind}_H^G(P), M \rangle} =  \dim_{\bar{\FF}_p} \Hom_{\bar{\FF}_p[G]}\left( \bar{\FF}_p[G] \otimes_{\bar{\FF}_p[H]} P, M \right)\\
& =  \dim_{\bar{\FF}_p} \Hom_{\bar{\FF}_p[H]} \left( P, \mathrm{Res}_H^G(M)\right) = \langle P, \mathrm{Res}^G_H(M)\rangle.
\end{aligned}\]
Hence (\ref{zwei}) is equal to
\begin{equation}
\begin{aligned}\label{drei}
&{- \frac{1}{n} \sum_{P \in \bar{X}} \sum_{d=1}^{e_P^\mathrm{t} -1} e_P^\mathrm{w} \, d \left\langle \mathrm{Cov}(\chi_P^d), \mathrm{Res}^G_{I_P}(\eta_p j_p \chi)\right\rangle }
 +  r(1-g_{Y_k} ) - |\bar{Y}^\mathrm{w}|.
\end{aligned}
\end{equation}
It is well-known that, if $R$ is the group ring of a finite group over a field and if $J$ denotes the Jacobson radical of $R$, then $JM=0$ and $\mathrm{Cov}(M)/J \mathrm{Cov}(M) \cong M$ for any simple $R$-module $M$. As the characters $\mathrm{Res}^G_{I_P}(\eta_p j_p \chi)$ and $\chi_P^d$ are simple, we hence obtain
\[\begin{aligned}
\MoveEqLeft {\left\langle \mathrm{Cov}(\chi_P^d), \mathrm{Res}_{I_P}^G(\eta_p j_p \chi) \right\rangle}\\
&=\dim_{\bar{\FF}_p} \Hom_{\bar{\FF}_p[I_P]} \left(\mathrm{Cov}(\chi_P^d), \mathrm{Res}_{I_P}^G(\eta_p j_p \chi) \right)\\
&= \dim_{\bar{\FF}_p} \Hom_{\bar{\FF}_p[I_P]} \left(\chi_P^d, \mathrm{Res}_{I_P}^G(\eta_p j_p \chi) \right)\\
&=\left\{ \begin{array}{ll} 1 & \mathrm{if} \quad \chi_P^d = \mathrm{Res}_{I_P}^G(\eta_p j_p \chi)\\ 0 & \mathrm{else} \end{array}\right.
\end{aligned}\]
for all $P \in \bar{X}$ and $d \in \ZZ$.

Let $P \in \bar{X}$. As the first ramification group $G_{P,1}$ of $\bar{\pi}$ at $P$ is a $p$-group, every character from $I_P$ to $\bar{\FF}_p^\times$ factorises modulo $G_{P,1}$. The induced character $\bar{\chi}_P: I_P/G_{P,1} \rightarrow \bar{\FF}_p^\times$ is injective by Proposition~IV.2.7 in \cite{Se} and hence generates the group $\Hom(I_P/G_{P,1}, \bar{\FF}_p^\times)$. In particular, there is unique integer $d(\chi_P, \chi) \in \{0, \ldots, e_P^\mathrm{t}-1\}$ such that
\[\chi_P^{d(\chi_P,\chi)} = \mathrm{Res}_{I_P}^G(\eta_p j_p \chi).\]
For instance, $d(\chi_P, \chi) = 0$ if $e_P^\mathrm{t} = 1$.

All this allows us to rewrite (\ref{drei}) in the following way:
\begin{equation}
\begin{aligned}\label{vier}
&{-\frac{1}{n} \sum_{P \in \bar{X}} e_P^\mathrm{w} \, d(\chi_P, \chi)}
 + r(1 - g_{Y_k}) - |\bar{Y}^\mathrm{w}|.
\end{aligned}
\end{equation}
For each point $\mathfrak{p} \in X$, we have a canonical bijection between the set
\[\Hom_{\FF_q}(k(\mathfrak{p}), \bar{\FF}_p)\]
of $\FF_q$-embeddings of the residue field $k(\mathfrak{p})$ into $\bar{\FF}_p$ and the fibre in $\bar{X}$ above~$\mathfrak{p}$.
Then, if $P \in \bar{X}$ lies above $\mathfrak{p} \in X$ and $\alpha: k(\mathfrak{p}) \rightarrow \bar{\FF}_p$ corresponds to~$P$,  we have $I_P = I_\mathfrak{p}$ and the character $\chi_P$ is equal to the composition
\[\xymatrix{I_{\mathfrak{p}} \ar[r]^{\chi_\mathfrak{p}} & k(\mathfrak{p})^\times \ar[r]^\alpha & \bar{\FF}_p^\times}.\] Hence (\ref{vier}) is equal to
\begin{equation}
\begin{aligned}\label{fuenf}
&{- \frac{1}{n} \sum_{\mathfrak{p} \in X} \; \sum_{\alpha \in \Hom_{\FF_q}(k(\mathfrak{p}), \bar{\FF}_p)} e_\mathfrak{p}^\mathrm{w} \, d(\alpha \circ \chi_\mathfrak{p}, \chi)}
+ r(1- g_{Y_k}) - \sum_{\mathfrak{q} \in Y^\mathrm{w}} \mathrm{deg}(\mathfrak{q}).
\end{aligned}
\end{equation}

For each $\sigma \in G$ and $\mathfrak{p} \in X$, the character $\chi_{\sigma(\mathfrak{p})}$ is equal to the composition
\[\xymatrix{I_{\sigma(\mathfrak{p})} \ar[r]^\sim & I_\mathfrak{p} \ar[r]^{\chi_\mathfrak{p}} & k(\mathfrak{p})^\times \ar[r]^\sigma & k(\sigma(\mathfrak{p}))^\times}\]
where the first map is given by conjugation with $\sigma$ and the last map is the inverse of the isomorphism induced by the local homomorphism $\sigma^{\#}: \mathcal{O}_{X,\sigma(\mathfrak{p})} \rightarrow \mathcal{O}_{X, \mathfrak{p}}$. It is also easy to see that the character $\mathrm{Res}^G_{I_{\sigma(\mathfrak{p})}}(\eta_p j_p \chi)$ is equal to the composition
\[\xymatrix{I_{\sigma(\mathfrak{p})} \ar[r]^\sim & I_\mathfrak{p} \ar@{->}[rrr]^{\mathrm{Res}_{I_{\mathfrak{p}}}^G(\eta_p j_p \chi)} &&& \bar{\FF}_p^\times}\]
where again the first map is given by conjugation with $\sigma$. In particular we have
\[d(\alpha' \circ \chi_{\sigma(\mathfrak{p})}, \chi) = d(\alpha' \circ \sigma \circ \chi_\mathfrak{p}, \chi)\]
for every $\alpha' \in \Hom_{\FF_q}(k(\sigma(\mathfrak{p})), \bar{\FF}_p)$. Furthermore, if $\alpha'$ runs through the set $\Hom_{\FF_q}(k(\sigma(\mathfrak{p})), \bar{\FF}_p)$, then $\alpha = \alpha' \circ \sigma$ runs through $\Hom_{\FF_q}(k(\mathfrak{p}), \bar{\FF}_p)$.

For each $\mathfrak{q} \in Y$, we have $\pi^{-1}(\mathfrak{q}) = \{ \sigma (\tilde{\mathfrak{q}}): \sigma \in G/G_{\tilde{\mathfrak{q}}} \}$ and hence $|\pi^{-1}(\mathfrak{q})| = \frac{n}{e_\mathfrak{q} f_\mathfrak{q}}$ where $f_\mathfrak{q} = f_{\tilde{\mathfrak{q}}} = [k(\tilde{\mathfrak{q}}): k(\mathfrak{q})]$ denotes the inertia degree of $\pi$ at $\tilde{\mathfrak{q}}$. Furthermore, we recall that $d(\alpha \circ \chi_{\tilde{q}}, \chi) = 0$ if $\fq \not\in Y^\mathrm{t}$ (by Lemma~\ref{dichotomy}). Thus (\ref{fuenf}) is equal to
\begin{equation}
\begin{aligned}\label{sechs}
&{- \sum_{\mathfrak{q} \in Y^\mathrm{t}} \; \sum_{\alpha \in \Hom_{\FF_q}(k(\tilde{\mathfrak{q}}), \bar{\FF}_p)} \frac{1}{e_\mathfrak{q} f_\mathfrak{q}} \; d(\alpha \circ \chi_{\tilde{\mathfrak{q}}}, \chi)}
 + r(1 - g_{Y_k}) - \sum_{\mathfrak{q} \in Y^\mathrm{w}} \mathrm{deg}(\mathfrak{q}).
\end{aligned}
\end{equation}

Now let $\fq \in Y^\mathrm{t}$. The norm residue homomorphism induces a surjective homomorphism
\[\gamma_\mathfrak{q}: k(\mathfrak{q})^\times \rightarrow I_{\tilde{\mathfrak{q}}},\]
see Section~\ref{ClassFieldTheory}. We recall that this implies that $e_\mathfrak{q}$ divides $|k(\mathfrak{q})| - 1$ and that the multiplicative group $k(\mathfrak{q})^\times$ contains all $e_\mathfrak{q}^\mathrm{th}$ roots of unity. As in Section~\ref{ClassFieldTheory}, we now consider the character $\chi_{\tilde{\mathfrak{q}}}: I_{\tilde{\mathfrak{q}}} \rightarrow k(\tilde{\mathfrak{q}})^\times$ as a map
\[\chi_{\tilde{\mathfrak{q}}}: I_{\tilde{\mathfrak{q}}} \rightarrow k(\mathfrak{q})^\times\]
from $I_{\tilde{\mathfrak{q}}}$ to the subgroup $k(\mathfrak{q})^\times$ of $k(\tilde{\mathfrak{q}})^\times$. All $\FF_q$-embeddings $\alpha: k(\tilde{\mathfrak{q}}) \hookrightarrow \bar{\FF}_p$ which extend a given $\FF_q$-embedding $\beta:k(\mathfrak{q}) \hookrightarrow \bar{\FF}_p$ therefore yield the same composition with $\chi_{\tilde{\mathfrak{q}}}$. For each $\beta$ there are $f_\mathfrak{q}$ such extensions $\alpha$. Therefore (\ref{sechs}) is equal to
\begin{equation}\label{sieben}
\begin{aligned}
& -\sum_{\mathfrak{q} \in Y^\mathrm{t}} \frac{1}{e_\mathfrak{q}} \; \sum_{\beta \in \Hom_{\FF_q}(k(\mathfrak{q}), \bar{\FF}_p)} d(\beta \circ \chi_{\tilde{\mathfrak{q}}}, \chi)
+ r(1- g_{Y_k}) - \sum_{\mathfrak{q} \in Y^\mathrm{w}} \mathrm{deg}(\mathfrak{q}).
\end{aligned}
\end{equation}

For each $\mathfrak{q} \in Y^\mathrm{t}$ let $\beta_\mathfrak{q}$ denote the embedding
\[\beta_\fq: \xymatrix{k(\fq) \ar@{^(->}[r] & k(\tilde{\fq}) \ar[r]^{\alpha_{\tilde{\fq}}} & \bar{\FF}_p}.\]
All other embeddings of $k(\fq)$ into $\bar{\FF}_p$ are then given by
\[\xymatrix{k(\mathfrak{q}) \ar[r]^{F^i} & k(\mathfrak{q}) \ar[r]^{\beta_\mathfrak{q}} & \bar{\FF}_p}, \quad i=1, \ldots, \deg(\fq),\]
where $F: k(\mathfrak{q}) \rightarrow k(\mathfrak{q})$, $x \mapsto x^q$, denotes the Frobenius automorphism of $k(\mathfrak{q})$ over $\FF_q$. Thus we can rewrite (\ref{sieben}) as
\begin{equation}\label{acht}
\begin{aligned}
& - \sum_{\mathfrak{q} \in Y^\mathrm{t}} \frac{1}{e_\mathfrak{q}} \sum_{i=1}^{\deg(\fq)} d((\beta_\mathfrak{q} \circ \chi_{\tilde{\mathfrak{q}}})^{q^i}, \chi)
+ r(1- g_{Y_k}) - \sum_{\mathfrak{q} \in Y^\mathrm{w}} \mathrm{deg}(\mathfrak{q}).
\end{aligned}
\end{equation}

For a moment let $\chi_0 = \beta_\mathfrak{q} \circ \chi_{\tilde{\mathfrak{q}}}$, $\chi_1 = \mathrm{Res}_{I_{\tilde{\mathfrak{q}}}}^G(\eta_p j_p \chi)$ and $d= d(\beta_\mathfrak{q} \circ \chi_{\tilde{\mathfrak{q}}}, \chi)$. The identity $\chi_0^d = \chi_1$ defining $d$ implies
\[(\chi_0^{q^i})^{dq^{\deg(\fq) -i}} = \chi_1^{q^{\deg(\fq)}} = \chi_1^{|k(\mathfrak{q})|} = \chi_1\]
because $e_\mathfrak{q}$ divides $|k(\mathfrak{q})| - 1$. Thus we obtain
\[d((\beta_\mathfrak{q} \circ \chi_{\tilde{\mathfrak{q}}})^{q^i}, \chi) \equiv d(\beta_\mathfrak{q} \circ \chi_{\tilde{\mathfrak{q}}}, \chi) q^{\deg(\fq) - i} \quad \textrm{mod} \quad e_\mathfrak{q}.\]
When $i$ runs from $1$ to $\deg(\fq)$, then $\deg(\fq) - i$ runs from $\deg(\fq) -1 $ to $0$. Finally, if $a$ and $m$ are any positive integers, then the residue class of $a$ modulo $m$ in $\{0, \ldots, m-1\}$ is given by $m\{\frac{a}{m}\}$. Therefore (\ref{acht}) is equal to
\begin{equation}\label{neun}
\begin{aligned}
&- \sum_{\mathfrak{q} \in Y^\mathrm{t}} \sum_{i=0}^{\deg(\fq)-1} \left\{\frac{d(\beta_\mathfrak{q} \circ \chi_{\tilde{\mathfrak{q}}}, \chi) q^i}{e_\mathfrak{q}} \right\}
+ r(1- g_{Y_k}) - \sum_{\mathfrak{q} \in Y^\mathrm{w}} \mathrm{deg}(\mathfrak{q}).
\end{aligned}
\end{equation}

As obviously $d(\beta_\fq \circ \chi_{\tilde{\fq}}, \chi) = d_\fq(\chi)$ for every $\fq \in Y^\mathrm{t}$, the term (\ref{neun}) is equal to the right-hand side of (\ref{Formula}) and the proof of Theorem~\ref{MainFormula} is finished.
\end{proof}

\begin{rmk} {\rm
The proof above works not only for $\bar{D}^\mathrm{w}$, but for any divisor $\bar{D} = \sum_{P \in \bar{X}} n_p [P]$ on $\bar{X}$ such that $n_P \equiv -1$ mod $e_P^\mathrm{w}$ and then yields the following more general result:
\begin{equation*}
\begin{aligned}
 \langle e(\psi(G, \bar{X}, \bar{D})), j_p \chi \rangle
= r(1- g_{Y_k}) - \sum_{\fq \in Y^\mathrm{t}} \sum_{i=0}^{\deg(\fq) - 1} g_i(l_\fq,\chi) 
+ \sum_{\fq \in Y} \deg(\fq)m_\mathfrak{q};
\end{aligned}
\end{equation*}
here, $l_P$ and $m_P$ are given by the equation $n_P = (e_P^\mathrm{w} - 1) + (l_P + m_P e_P^\mathrm{t}) e_P^\mathrm{w}$ and $g_i(l_\fq,\chi)$ denotes the unique rational number in the interval $[-\frac{l_\fq}{e_\fq}, 1-\frac{l_\fq}{e_\fq}[$ that is congruent to $\frac{d_\fq(\chi) q^i}{e_\fq}$ modulo $\ZZ$. For details see the first arXiv version of this paper. }
\end{rmk}

For $\fq \in Y^\mathrm{t}$, let the Gauss sum $\tau(\chi_{k(\fq)}) \in \bar{\QQ}$ be defined as in the previous section.

\begin{cor}\label{q=p}
If $q=p$, we have
\[\begin{aligned}
\MoveEqLeft \langle e(\psi(G, \bar{X}), j_p\chi\rangle
=  \; r(1- g_{Y_k}) - \sum_{\fq \in Y^\mathrm{t}} v_p(j_p(\tau(\chi_{k(\fq)})))- \sum_{\fq \in Y^\mathrm{w}} \deg(\fq).
\end{aligned}\]
\end{cor}

\begin{proof}
This follows immediately from Theorem~\ref{MainFormula} (see also (\ref{neun})) and Theorem~\ref{GaussSum}.
\end{proof}

\section{The Main Theorem}\label{MainTheoremSection}

We keep the assumptions about $k$, $X$ and $G$ introduced in the previous section and continue to assume that $\pi$ is at most weakly ramified, but we no longer assume that $G$ is abelian.

The object of this section is to prove our `strong' formula stated in the introduction and in  Theorem~\ref{FormulationMainTheorem} below; it relates the epsilon constants associated with $X$ and finite-dimensional complex representations of $G$ to an equivariant Euler characteristic of $\bar{X} := X \times_{\FF_p} \bar{\FF}_p$. The main ingredients in its proof are Artin's induction theorem, Corollary~\ref{q=p} and Corollary~\ref{EpsilonConstant}.

\begin{rmk}{\em
Note that, in this section, our base field is $\FF_p$ rather than $\FF_q$ or $k$. Euler characteristics of the geometrically irreducible curve $X/k$ (or of $X \times_{k} \bar{\FF}_p/\bar{\FF}_p)$ are finer invariants than the corresponding Euler characteristics of $X/\FF_p$ and, as already explained in Remark~5.3 of \cite{Ch}, it seems not to be possible to relate these finer invariants to epsilon constants associated with $X$ and~$G$. In this paper, this becomes apparent in the transition from Theorem~\ref{MainFormula} to Corollary~\ref{q=p}. More concretely, if for instance $\fq \in Y^\mathrm{t}$ is a $k$-rational point, the sum $\sum_{i=0}^{\deg(\fq)-1} \left\{\frac{d_\fq(\chi) q^i}{e_\fq}\right\}$ occurring in (\ref{neun}) is reduced to $\frac{d_\fq(\chi)}{e_\fq}$ and seems not to be related to the $p$-adic valuation of the corresponding Gauss sum.
}\end{rmk}

We use notations similar to those introduced in the previous chapter; most notably, the element $\psi(G,\bar{X}) \in K_0(\bar{\FF}_p[G])$ is defined by the equation
\[c(\psi(G, \bar{X})) = [H^0(\bar{X}, \cO_{\bar{X}}(\bar{D}^\mathrm{w}))] - [H^1(\bar{X}, \cO_{\bar{X}}(\bar{D}^\mathrm{w}))] \quad \textrm{ in } \quad K_0(G,\bar{\FF}_p)\]
where $\bar{D}^\mathrm{w}$ is the divisor
\[\bar{D}^\mathrm{w} = \bar{D}^\mathrm{w}(G,X) = - \sum_{P \in \bar{X}^\mathrm{w}}[P].\]
For every representation $V$ of $G$ and for every subgroup $H$ of $G$, let $V^H$ denote the subspace of $V$ fixed by $H$.

Let $V$ be a $\bar{\QQ}$-representation of $G$. After applying $j_p:\bar{\QQ} \hookrightarrow \bar{\QQ}_p$ to epsilon constant $\varepsilon(V)$, we may consider its $p$-adic valuation $v_p(j_p(\varepsilon(V)))$. The following theorem computes this $p$-adic valuation in terms of the element $\psi(G, \bar{X})$ introduced above.

\begin{thm}\label{FormulationMainTheorem}
For every finite-dimensional $\bar{\QQ}$-representation $V$ of $G$ we have
\begin{equation}\label{MainTheorem}
- v_p(j_p(\varepsilon(V))) = \langle e(\psi(G,\bar{X})), j_pV\rangle + \sum_{Q \in \bar{Y}^\mathrm{w}} \dim_{\bar{\QQ}}(V^{G_{\tilde{Q}}}).
\end{equation}
In particular, the $p$-adic valuation of $j_p(\varepsilon(V))$ is an integer.
\end{thm}

\begin{rmk}{\em
Using the notation $E(G,X)$ introduced in the next section, Theorem~\ref{FormulationMainTheorem} can obviously be reformulated in the following succinct way.}

We have:
\[E(G,X) = e(\psi(G,\bar{X})) + \sum_{Q \in \bar{Y}^\mathrm{w}} \left[\mathrm{Ind}_{G_{\tilde{Q}}}^G ({\bf 1})\right] \quad \textrm{in} \quad K_0(\bar{\QQ}_p[G])_\QQ,\]
where $\bf 1$ means the trivial representation of rank one. In particular, $E(G,X)$ belongs to the integral part $K_0(\bar{\QQ}_p[G])$ of $K_0(\bar{\QQ}_p[G])_\QQ$.
\end{rmk}

\begin{proof}
By Artin's induction theorem \cite[Corollaire of Th\'eor\`eme 17]{Se2}, every element of $K_0(\bar{\QQ}[G])_\QQ$ can be written as a rational linear combination of representations induced from multiplicative characters of cyclic subgroups of~$G$. It therefore suffices to prove the following three statements.
\begin{enumerate}
\item[(a)] When applied to the direct sum $V_1 \oplus V_2$ of two finite-dimensional $\bar{\QQ}$-representations $V_1$ and $V_2$ of $G$, each side of equation (\ref{MainTheorem}) is equal to the sum of the corresponding values for $V_1$ and $V_2$.
\item[(b)] Each side of equation~(\ref{MainTheorem}) is invariant under induction with respect to~$V$.
\item[(c)] Equation~(\ref{MainTheorem}) is true if $G$ is a cyclic group and $V$ corresponds to a multiplicative character $\chi$ of $G$.
\end{enumerate}
Statement~(a) is obvious for the right-hand side of (\ref{MainTheorem}). For the left-hand side, this follows from (5.2) and (5.11.2) in \cite{De}.

It is well-known that the L-function $L(V,t)$ is invariant under induction with respect to $V$. This immediately implies that the left-hand side of (\ref{MainTheorem}) is invariant under induction. We now prove that the right-hand side of (\ref{MainTheorem}) is invariant under induction as well. Let $H$ be a subgroup of $G$ and let $\bar{\pi}_H: \bar{X} \rightarrow \bar{X}/H$ denote the corresponding projection. Furthermore, let $W$ be a finite-dimensional $\bar{\QQ}$-representation of~$H$. Then the right-hand side of (\ref{MainTheorem}) applied to $V=\mathrm{Ind}_H^G(W)$ is equal to
\begin{equation}
\langle e(\psi(G, \bar{X})), j_p \mathrm{Ind}_H^G(W) \rangle + \sum_{Q \in \bar{Y}^\mathrm{w}} \langle \mathbf{1}, \mathrm{Res}^G_{G_{\tilde{Q}}}(\mathrm{Ind}^G_H(W))\rangle
\end{equation}
where $\mathbf{1}$ means the trivial representation of dimension~1. By Frobenius reciprocity \cite[Th\'eor\`eme~13]{Se2}, this is equal to
\begin{equation}\label{22}
\langle e(\mathrm{Res}^G_H(\psi(G, \bar{X}))), j_p W \rangle + \sum_{Q \in \bar{Y}^\mathrm{w}} \langle \mathrm{Res}^G_H(\mathrm{Ind}^G_{\tilde{Q}}(\mathbf{1})), W \rangle.
\end{equation}
For each $Q \in \bar{Y}$ we have a canonical bijection
\[ \{H\sigma G_{\tilde{Q}}: \sigma \in G \} \rightarrow \bar{\pi}_H(\bar{\pi}^{-1}(Q)), \quad H\sigma G_{\tilde{Q}} \mapsto \bar{\pi}_H(\sigma(\tilde{Q})).\]
Therefore, Mackey's double coset formula \cite[Theorem~44.2]{CR1} implies that (\ref{22}) is equal to
\begin{equation}\label{23}
\begin{aligned}
\MoveEqLeft \langle e(\mathrm{Res}^G_H(\psi(G, \bar{X})), j_p W \rangle + \sum_{R \in \bar{\pi}_H(\bar{X}^\mathrm{w})} \langle \mathrm{Ind}^H_{H_{\tilde{R}}}(\mathbf{1}), W \rangle\\
&= \langle e(\mathrm{Res}^G_H(\psi(G, \bar{X})), j_p W \rangle + \sum_{R \in \bar{\pi}_H(\bar{X}^\mathrm{w})} \dim_{\bar{\QQ}}(W^{H_{\tilde{R}}}),
\end{aligned}
\end{equation}
where, of course, $\tilde{R}$ is a point in $\bar{\pi}_H^{-1}(R)$ and $H_{\tilde{R}} := G_{\tilde{R}} \cap H$. Let $\mathcal{S}$ denote the set of all points $R \in \bar{X}/H$ such that $\bar{\pi}: \bar{X} \rightarrow \bar{X}/G = Y$ is wildly ramified at $\tilde{R}$, but $\bar{\pi}_H: \bar{X} \rightarrow \bar{X}/H$ is tamely ramified at $\tilde{R}$. Then for each $R \in \mathcal{S}$, the coefficient of the divisor $\bar{D}^\mathrm{w}$ at $\tilde{R}$ is $-1=(e_R(H)-1) + (-1) e_R(H)$, where $e_R(H)$  denotes the ramification index of $\bar{\pi}_H$ at $\tilde{R}$. Using \cite[Theorem~4.5]{Ko}, we therefore obtain:
\begin{equation*}
\begin{aligned}
\MoveEqLeft \mathrm{Res}_H^G(\psi(G, \bar{X})) 
= \psi(H, \bar{X}) + \sum_{R \in \mathcal{S}} \left(\sum_{d=1}^{e_R(H)-1}\left[\mathrm{Ind}_{H_{\tilde{R}}}^H(\chi_{\tilde{R}}^{-d})\right] - \left[\bar{\FF}_p[H]\right]\right)\\
& = \psi(H, \bar{X}) - \sum_{R \in \mathcal{S}} \left[\mathrm{Ind}^H_{H_{\tilde{R}}}(\mathbf{1})\right],
\end{aligned}
\end{equation*}
where $\chi_{\tilde{R}}$ also denotes the restriction of $\chi_{\tilde{R}}: G_{\tilde{R}} \rightarrow \bar{\FF}_p$ to $H_{\tilde{R}}$. Therefore (\ref{23}) is equal to
\begin{equation}
\begin{aligned}
\MoveEqLeft \langle e(\psi(H, \bar{X})), j_p W \rangle - \sum_{R \in \mathcal{S}} \dim_{\bar{\QQ}}(W^{H_{\tilde{R}}}) + \sum_{R \in \bar{\pi}_H(\bar{X}^\mathrm{w})} \dim_{\bar{\QQ}} (W^{H_{\tilde{R}}})\\
&= \langle e(\psi(H, \bar{X})), j_p W \rangle + \sum_{R \in (\bar{X}/H)^\mathrm{w}}\dim_{\bar{\QQ}}(W^{H_{\tilde{R}}}),
\end{aligned}
\end{equation}
where $(\bar{X}/H)^\mathrm{w}$ denotes the set of all $R \in \bar{X}/H$ such that $\bar{\pi}_H$ is not tamely ramified at $\tilde{R}$. This finishes the proof of statement~(b).

We finally prove statement~(c). Let $G$ be cyclic and let $\chi$ be a multiplicative character of $G$. Let $r$ denote the degree of $k$ over $\FF_p$. By \cite[Remark~5.4]{Ch}, the epsilon constant of $X/\FF_p$ is the same as the epsilon constant of $X$ considered as a scheme over $k$.  Hence, by Corollary~\ref{EpsilonConstant} we have
\begin{equation}\label{24}
 - v_p(j_p(\varepsilon(\chi))) = r(1- g_{Y_k}) - \sum v_p(j_p(\tau(\chi_{k(\fq)}))) - \sum \deg(\fq).
\end{equation}
Here the first sum runs over all $\fq \in Y$ such that $\chi$ is tamely ramified at $\fq$ and the second sum runs over all $\fq \in {Y}$ such that $\chi$ is not tamely ramified at~$\tilde{q}$. These two sets differ from $Y^\mathrm{t}$ and $Y^\mathrm{w}$ by the set of those $\fq \in \bar{Y}^\mathrm{w}$ such that $\chi$ vanishes on $G_{\tilde{\fq},1}$. For such $\fq$ we have $G_{\tilde{\fq}} = G_{\tilde{\fq},1}$ by Lemma~\ref{dichotomy} and the corresponding Gauss sum is trivial. Hence (\ref{24}) is equal to
\begin{equation}
r(1-g_{Y_k}) - \sum_{\fq \in Y^\mathrm{t}} v_p(j_p(\tau(\chi_\fq))) - \sum_{\fq \in Y^\mathrm{w}} \deg(\fq) + \sum_{\fq \in Y^\mathrm{w}: \mathrm{Res}^G_{G_{\tilde{\fq}}}(\chi) = \mathbf{1}} \deg(\fq)
\end{equation}
which in turn is equal to
\begin{equation}
\langle e(\psi(G, \bar{X})), j_p \chi \rangle + \sum_{Q \in \bar{Y}^\mathrm{w}} \langle \mathbf{1}, \mathrm{Res}^G_{G_{\tilde{Q}}} (\chi) \rangle
\end{equation}
by Corollary~\ref{q=p}, as was to be shown.
\end{proof}

\section{A Weak, but General Formula}\label{GeneralFormula}

Let $k$, $X$ and $G$ be as in the previous section.
Without assuming any condition on the type of ramification of the associated projection $\pi: X \rightarrow Y$ we give, in Theorem~\ref{WeakTheorem} below, a relation between the equivariant Euler characteristic $\chi(G,\bar{X},\cO_{\bar{X}})$ of $\bar{X} := X \times_{\FF_p} \bar{\FF}_p$ and epsilon constants associated with~$X$ and finite-dimensional complex representations of $G$. After reducing to the so-called domestic case (i.e., when $p$ does not divide the order of $G$) using Artin's induction theorem for modular representation theory, this relation follows from Theorem~5.2 in \cite{Ch}. We call this relation `weak' because it does not capture any equivariant information given by the $p$-part of $G$.

We begin by explaining how epsilon constants define a natural element $E(G,X)$ in $K_0(\bar{\QQ}_p[G])_\QQ$.

\begin{definition}{\em
By \cite[(5.2)]{De}, associating with every finite-dimensional $\bar{\QQ}$-representation $V$ of $G$ the $p$-adic valuation $v_p(j_p(\varepsilon(V)))$ of the epsilon constant $j_p(\varepsilon(V))$ defines a homomorphism from $K_0(\bar{\QQ}[G])_\QQ := K_0(\bar{\QQ}[G]) \otimes \QQ$ to $\QQ$. As the classical character pairing
$\langle \hspace{1em}, \hspace{1em}\rangle: K_0(\bar{\QQ}_p[G]) \times K_0(\bar{\QQ}_p[G]) \rightarrow \ZZ$
is non-degenerate, there is a unique element $E(G,X) \in K_0(\bar{\QQ}_p[G])_\QQ$ such that
\[\langle E(G,X), j_p(V) \rangle = - v_p(j_p(\varepsilon(V))) \]
for all finite-dimensional $\bar{\QQ}$-representations $V$ of $G$. We call $E(G,X)$ the {\em (virtual) epsilon representation associated with the action of $G$ on $X$}.
}\end{definition}

It follows for instance from the definition of $L(V,t)$ that, for every $\alpha \in \mathrm{Aut}(\bar{\QQ})$, we have $\varepsilon(\alpha(V)) = \alpha(\varepsilon(V))$ and that therefore $E(G,X)$ does not depend on the embedding $j_p$. Furthermore, it follows from Frobenius reciprocity and from the definition of Artin L-functions and epsilon constants that, when restricted to a subgroup~$H$ of~$G$, the element $E(G,X)$ becomes $E(H,X)$.

Recall that
\[d: K_0(\bar{\QQ}_p[G]) \rightarrow K_0(G, \FF_p)\]
denotes the decomposition map.

\begin{thm}\label{WeakTheorem}
We have
\begin{equation}
d(E(G,X)) = \chi(G,\bar{X}, \cO_{\bar{X}}) \quad \textrm{ in } \quad K_0(G,\bar{\FF}_p)_\QQ.
\end{equation}
In particular, $d(E(G,X))$ lies in the integral part $K_0(G, \bar{\FF}_p)$ of $K_0(G,\bar{\FF}_p)_\QQ$.
\end{thm}

\begin{proof}
As the canonical pairing
$\langle \hspace{1em}, \hspace{1em} \rangle: K_0(\bar{\FF}_p[G]) \times K_0(G, \bar{\FF}_p) \rightarrow \ZZ$
(see the beginning of the proof of Theorem~\ref{MainFormula}) is non-degenerate as well \cite[\S 14.5(b)]{Se2}, it suffices to show that
\begin{equation}\label{Characterequation}
\langle P, d(E(G,X)) \rangle = \langle P, \chi(G,\bar{X}, \cO_{\bar{X}}) \rangle
\end{equation}
for all finitely generated projective $\bar{\FF}_p[G]$-modules $P$. By Artin's induction theorem for modular representation theory \cite[Th\'eor\`eme~40]{Se2}, every element in $K_0(\bar{\FF}_p[G])_\QQ$ can be written as a rational linear combination of representations induced from one-dimensional projective representations of cyclic subgroups of~$G$. Furthermore, by Frobenius reciprocity, both sides of (\ref{Characterequation}) are invariant under induction with respect to $P$. As in the proof of Theorem~\ref{FormulationMainTheorem}, we may therefore assume that $G$ is cyclic and that $P$ corresponds to a character $\chi:G \rightarrow \bar{\FF}_p^\times$. The fact that $P$ is projective moreover implies that $p$ does not divide the order of $G$. In particular, the projection $\pi$ is tamely ramified and we conclude from Theorem~\ref{FormulationMainTheorem} (or actually already from Theorem~5.2 in \cite{Ch}) that $E(G,X) = e (\psi(G, \bar{X}))$. We therefore have
\[d(E(G,X)) = d(e(\psi(G, \bar{X}))) = c( \psi(G, \bar{X})) = \chi(G,\bar{X}, \cO_{\bar{X}}),\]
as was to be shown.
\end{proof}

\begin{rmk}{\em
If $\pi$ is weakly ramified, Theorem~\ref{WeakTheorem} can also be derived from Theorem~\ref{FormulationMainTheorem}. Details are contained in the first arXiv version of this article.
}\end{rmk}

We end with the following problem.

\begin{problem} {\em
Describe $E(G,X)$ within $K_0(\bar{\QQ}_p[G])_\QQ$ in terms of global geometric invariants of $\bar{X}$ in a way that generalises Theorem~\ref{FormulationMainTheorem} from the weakly ramified situation to the general situation considered in this section.
}\end{problem}

\vspace*{1em}

\begin{center}
{\bf \large Appendix\\
\vspace*{0.2em}
Integrality of epsilon representations\\ of wildly ramified Galois covers}\\
\vspace*{0.5em}
{\sc Bernhard K\"ock} and {\sc Adriano Marmora}
\end{center}
\nopagebreak \vspace{1em}\nopagebreak
Let $q=p^f$ with $f\ge1$, let $X$ be a geometrically irreducible smooth projective curve over $\FF_q$ and let $G$ be a finite subgroup of $\mathrm{Aut}(X/\FF_q)$. Theorem~4.2 \pagebreak states that the corresponding  epsilon representation $E(G,X)$ (introduced in Definition 5.1) is integral, i.e.\ belongs to the lattice $K_0(\bar{\QQ}_p[G])$ in $K_0(\bar{\QQ}_p[G])_\QQ$, if the canonical projection $\pi:X \rightarrow Y:=X/G$ is at most weakly ramified. Theorem~5.2 states that its image $d(E(G,X))$ in $K_0(G,\bar{\FF}_p)_\QQ$ {\em always} belongs to the lattice $K_0(G,\bar{\FF}_p)$.

In this appendix we study the question to which extent is $E(G,X)$ integral in the case $\pi$ is arbitrarily wildly ramified. We will show that $E(G,X)$ becomes integral after multiplying it by $2$ if $p=2$ and after multiplying it by $p-1$ if $p\ge3$, see Theorem~A.2. Furthermore we will give an example of a Galois cover~$\pi$ such that the denominator $p-1$ does occur in $E(G,X)$, see Proposition~A.3, and an another one such that the denominator $2$ does occur in $E(G,X)$ if $p=2$, see end of this appendix.\\

To begin with, we temporarily fix  a multiplicative character $\chi:G \rightarrow \bar{\QQ}^\times$ and consider the corresponding local epsilon constant $\varepsilon(\chi_{\tilde{\fq}}, \psi_\fq, {\rm d}x_\fq)$ at some $\fq \in Y$ (where $\chi_{\tilde{\fq}}:G_{\tilde{\fq}} \rightarrow \bar{\QQ}^\times$, $\psi_\fq : K_\fq \rightarrow \bar{\QQ}^\times$ and ${\rm d}x_\fq$ are defined as in Section~2, see the proof of and the two paragraphs before Theorem~2.3).  We recall that the Artin conductor $\mathrm{a}(\chi_{\tilde{\mathfrak{q}}})$ of $\chi$ at $\fq$ is the smallest integer $i \ge 0$ such that $\chi$ vanishes on the higher ramification group $G_{\tilde{\fq}}^i$ in upper numbering for one (any) point $\tilde{\fq} \in X$ above $\fq$. The following lemma generalises Lemma~2.6.\\

\noindent{\bf Lemma A.1}. {\em If $\mathrm{a}(\chi_{\tilde{\fq}}) \ge 2$ (i.e., $\chi$ is not tamely ramified at $\fq $), we have
\[\varepsilon(\chi_{\tilde{\fq}}, \psi_\fq, {\rm d}x_\fq)^2 \sim q^{\mathrm{a}(\chi_{\tilde{\fq}})};\]
in particular, the $p$-adic valuation $v_p(j_p(\varepsilon(\chi_{\tilde{\fq}}, \psi_\fq, {\rm d}x_\fq)))$ does not depend on the embedding $j_p:\bar{\QQ} \hookrightarrow \bar{\QQ}_p$ and is equal to $f \,\mathrm{a}(\chi_{\tilde{\fq}})/2$.}

\begin{proof}
If $p\geq 3$, this follows from Proposition~8.7(ii) in \cite{AbbesSaito}. To explain this and to be able to use the notation introduced in {\itshape loc.~cit.}, we temporarily redefine $\psi:=\psi_\fq$ and $\chi:= \underline{\chi}_{\tilde{\fq}}$ (where $\underline{\chi}_{\tilde{\fq}}:K_\fq^\times \rightarrow \bar{\QQ}^\times$ is defined as in the proof of Theorem~2.3). Using further notation introduced in {\itshape loc.~cit.}, we have $n+1=\mathrm{sw}(\chi)+1=\mathrm{a}(\chi)$,
$G^2_{\psi_k}\sim q$  (by {\itshape loc.~cit.}~(8.5.3)) and $\mathrm{ord(c)}=-\mathrm{a}(\chi)$ (by {\itshape loc.~cit.}~Section~8.6 and Proposition~8.7(ii); note that $\mathrm{ord}(\psi)=0$ by our assumption on $\psi$). Hence {\itshape loc.~cit.}~(8.7.3) implies
\begin{align*}
&\varepsilon(\chi,\psi)^2=\varepsilon_0(\chi,\psi)^2 \\
&\hspace*{3.4em} =\chi(c)^{-2}\psi(c)^2 q^{-2\mathrm{ord}(c)}(G^2_{\psi_k})^{-\mathrm{a}(\chi)}
 \sim q^{-2\mathrm{ord}(c)-\mathrm{a}(\chi)}=q^{\mathrm{a}(\chi)},
\end{align*}
as was to be shown.

If  $p=2$, Lemma~A.1 follows from Theorem~48.4 on p.~301 in \cite{Bushnell_Henniart}; note that $\varepsilon(\rho, 1/2, \psi) = \varepsilon(\rho, \psi) q^{-\frac{\mathrm{a}(\chi)}{2}}$ in {\em loc.\ cit.} since $\mathrm{ord}(\psi)=0$.
\end{proof}

\noindent {\bf Theorem A.2}.
The epsilon representation $E(G,X)$ belongs to $\frac{1}{2} K_0(\bar{\QQ}_p[G])$ if $p=2$ and to $\frac{1}{p-1}K_0(\bar{\QQ}_p[G])$ if $p\ge3$.

\begin{proof} By definition of $E(G,X)$, we need to show that, for every complex representation $V$ of $G$, the valuation $v_p(j_p(\varepsilon(V)))$ belongs to $\frac{1}{2} \ZZ$ if $p=2$ and to $\frac{1}{p-1} \ZZ$ if $p \ge 3$. As in the  proof of Theorem~4.2, using Brauer's induction theorem, we may assume that $V$ corresponds to a multiplicative character~$\chi$ of $G$.
As $\chi$ factorises via~$G^\mathrm{ab}$, we may furthermore assume that $G$ is abelian. By using Deligne's product formula, similarly to Corollary~2.7, we derive from Theorem~2.3 and Lemmas 2.4 and 2.5 the relation
\begin{equation}\label{EpsilonConstantFormulaAgain}
\varepsilon(\chi) \sim q^{g_Y -1} \cdot \prod \tau(\chi_{k(\fq)}) \cdot \prod \varepsilon(\chi_{\tilde{\fq}}, \psi_\fq, {\rm d}x_\fq)
\end{equation}
where the first product runs over all $\fq \in Y$ such that $\chi$ is tamely ramified at $\fq$ and the second product runs over all $\fq \in Y$ such that $\chi$ is not tamely ramified at $\fq$. By Stickelberger's formula (see the version given in the proof of Theorem~1.3), the valuations of the Gauss sums $\tau(\chi_{k(\fq)})$ belong to $\frac{1}{p-1}\ZZ$ and, by Lemma~A.1, the valuations of the local epsilon constants in the second product belong to $\frac{1}{2}\ZZ$. This proves Theorem~A.2 both when $p=2$ and when~$p \ge 3$.
\end{proof}

\begin{center}{{\bf First Example}} \end{center}

Over the next two or three pages, we construct an example of a Galois cover of curves, $\pi:X \rightarrow \PP^1_{\FF_q}$, for which the denominator $p-1$ does occur in $E(G,X)$. Let $L_\mathrm{t} =\FF_q(u)$ be the Kummer extension of the field of rational functions $\FF_q(t)$ obtained by adjoining a primitive $(q-1)^\textrm{th}$ root $u$ of $t$. Let $L_\mathrm{w}$ be the Artin-Schreier extension of $\FF_q(t)$ obtained by adjoining a root of the polynomial $T^p-T-t$. Finally, let $L$ denote the compositum of $L_\mathrm{t}$ and $L_\mathrm{w}$, i.e., the Artin-Schreier extension of $L_\mathrm{t}=\FF_q(u)$ obtained by adjoining a root of the polynomial $T^p-T-u^{q-1}$. Then $L_\mathrm{t}$, $L_\mathrm{w}$ and $L$ are cyclic Galois extensions of $\FF_q(t)$ with Galois groups $\FF_q^\times$, $\FF_p$ and $G:= \FF_q^\times \times \FF_p$, respectively. We now define $\pi: X \rightarrow \PP^1_{\FF_q}$ to be the Galois cover of smooth projective curves corresponding to the extension $L|\FF_q(t)$. It follows from Example~2.5 in \cite{Ko} that $\pi$ is not weakly ramified at $\infty$ if $q >2$.

We fix a non-trivial character $\varphi: \FF_p \rightarrow \bar{\QQ}_p^\times$ and let $\omega: \FF_q^\times \rightarrow \bar{\QQ}_p^\times$ denote the Teichm\"uller character. We write $\varphi$ and $\omega$ also for the characters of $G$ obtained by composing with the canonical projections $G\rightarrow \FF_p$ and $G \rightarrow \FF_q^\times$, respectively. The characters $\omega^i\varphi^j$, $i=0, \ldots, q-2$, $j=0, \ldots, p-1$, then form a complete system of irreducible representations of $G$ and hence a $\ZZ$-basis of $K_0(\bar{\QQ}_p[G])$.\\

\noindent {\bf Proposition A.3}. {\em We have
\[E(G,X)= -\sum_{i=1}^{q-2} \frac{s_p(i)}{p-1} \sum_{j=1}^{p-1} [\omega^i\varphi^j] \quad \textrm{ in } \quad K_0(\bar{\QQ}_p[G])_\QQ\]
where $s_p(i)$ denotes the sum of $p$-adic digits of $i$. In particular, the coefficient of $[\omega\varphi^j]$ is equal to $-\frac{1}{p-1}$ for every $j \in \{1, \ldots, p-1\}$.
}

\begin{proof} For brevity, we write $\chi^{i,j}$ for the unique character $G \rightarrow \bar{\QQ}^\times$ such that $j_p \circ \chi^{i,j} = \omega^i\varphi^j$. It suffices to show that the formula in Proposition~A.3 is true after pairing it with $[\omega^i\varphi^j]$ for $i=0, \ldots, q-2$, $j=0, \ldots, p-1$. By definition of $E(G,X)$, this amounts to showing the equation
\begin{equation}\label{ExamplePAdicValuation}
v_p(j_p(\varepsilon(\chi^{i,j}))) = \left\{\begin{aligned}
&0 && \textrm{ if } i=0 \textrm{ or } j=0,\\
&\frac{s_p(i)}{p-1} && \textrm{ else}.\end{aligned}\right.
\end{equation}
It is known and easy to see that the cover $\pi$ is ramified exactly over $0$ and $\infty$ with decomposition groups $\FF_q^\times$ and $G$, respectively. Similarly to formula~(\ref{EpsilonConstantFormulaAgain}) we therefore obtain
\[\varepsilon(\chi^{i,j}) \sim q^{-1}\; \tau(\chi^{i,j}_{k(0)}) \;\varepsilon(\chi^{i,j}_{\tilde{\infty}}, \psi_\infty, {\rm d}x_\infty) .\]

It follows from Satz~V(3.4) in \cite{Ne} (see also the proof of Proposition~\ref{composition}) that the homomorphism $\FF_q^\times = k(0)^\times \rightarrow G_{\tilde{0}}=\FF_q^\times$ induced by the norm residue symbol is the identity map. The Gauss sum $\tau(\chi^{i,j}_{k(0)})$ is therefore the Gauss sum corresponding to the character $\omega^i$ of $\FF_q^\times$ if $i\ \in \{1, \ldots, q-2\}$ and trivial if $i=0$. Hence its $p$-adic valuation is  $\frac{s_p(i)}{p-1}$ by Stickelberger's formula (see the version given in the proof of Theorem~1.3).

If $j=0$, the character $\chi^{i,j}$ is tamely ramified also above $\infty$. The homomorphism $\FF_q^\times= k(\infty)^\times \rightarrow G_{\tilde{\infty},0}/G_{\tilde{\infty},1} =\FF_q^\times$ induced by the norm residue symbol is the inverse map. Similarly to the previous paragraph, we therefore obtain $v_p(j_p(\varepsilon(\chi^{i,0}_{\tilde{\infty}}, \psi_\infty, {\rm d}x_\infty)))= v_p(j_p(\tau(\chi^{i,0}_{k(\infty)}))) = \frac{s_p(q-1-i)}{p-1}$. Using the elementary and easy-to-verify formula $s_p(i) + s_p(p^f-1-i) = (p-1)f$, we finally obtain $v_p(j_p(\varepsilon(\chi^{i,j}))) =0$ if $j=0$, as claimed.

In the case $j \in \{1, \ldots, p-1\}$, we first note that, above $\infty$, the first ramification group is $\FF_p$ and the smallest integer $i$ such that the $i^\textrm{th}$ higher ramification group of the sub-extension $L_\mathrm{w}/\FF_q(t)$ in lower or upper numbering vanishes is $2$ (see Example~2.5 in \cite{Ko}).
Hence, by Herbrand's Theorem (see Satz~II(10.9) in \cite{Ne}), the Artin conductor of $\chi^{i,j}$ at~$\infty$ is equal to~$2$ and the valuation of the local epsilon constant at $\infty$ is equal to $f$ by Lemma~A.1. Putting everything together we obtain $v_p(j_p(\varepsilon(\chi^{i,j}))) = \frac{s_p(i)}{p-1}$, as claimed.
\end{proof}

\textbf{{Remark A.4}}. The example constructed above has the following natural sheaf-theoretic interpretation and has actually been found this way. This interpretation moreover allows a cohomological proof of Proposition~A.3 which in particular is of geometric and global nature and does not use the product formula for epsilon constants. 
Although the rigid cohomology would be even more natural, we explain all this via the \'etale $\ell$-adic cohomology, for it is more widely known. We refer the reader to the \emph{expos\'e} [Sommes trig.] in \cite{DeligneSGAfour_and_half} or to \cite{Mi} for more details.

Let $\ell$ be a prime different from $p$ and let $j_\ell\colon\bar{\QQ} \hookrightarrow \bar{\QQ}_l$ be an embedding.
The character $\omega\colon \FF_q^\times \rightarrow \bar{\QQ}_p^\times$ (resp.
$\varphi\colon \FF_p \rightarrow \bar{\QQ}_p^\times$) factors through $j_p\colon\bar{\QQ} \hookrightarrow \bar{\QQ}_p$, so it makes sense to consider its composition with $j_\ell$,
which defines
a rank-one smooth $\ell$-adic sheaf   $\mathcal{K}$ on $\GG_m$ (resp., $\mathcal{L}$ on $\AA^1_{\FF_q}$), trivialized by the Kummer (resp., Artin--Schreier) covering.
By construction
$\mathcal{K}$ is tame at $0$ and $\infty$ and $\mathcal{L}$  is ramified at $\infty$ with Artin conductor $2$.
For $i \in \{0, \ldots, q-2\}$ and $j \in \{0, \ldots, p-1\}$, we set $\mathcal{M}^{i,j}:= \mathcal{K}^{\otimes i} \otimes \mathcal{L}^{\otimes j}|_{\GG_m}$ which is a smooth sheaf on $\GG_m$.
Then $\pi:X \rightarrow \PP^1_{\FF_q}$ is the smallest Galois cover of $\PP^1_{\FF_q}$, \'etale over $\GG_m$, such that  $\pi_{|\GG_m}^*\mathcal{M}^{1,1}$ is trivial.

Let us prove (\ref{ExamplePAdicValuation}).  The epsilon constant $j_\ell(\varepsilon(\chi^{i,j}))$ is equal to the global epsilon constant of the $\ell$-adic sheaf $u_*(\mathcal{M}^{i,j})$ where $u: \GG_m \hookrightarrow \PP^1_{\FF_q}$ denotes the inclusion, see~VI.13.6 in \cite{Mi}.
We first consider the main case $i\in \{1, \ldots, q-2\}$ and $j \in \{1, \ldots,  p-1\}$. We then have $u_*(\mathcal{M}^{i.j}) = u_!(\mathcal{M}^{i,j})$ because the character~$\omega^i\varphi^j$ is not trivial and ramified at $0$ and $\infty$, so
\[\varepsilon(u_*(\mathcal{M}^{i,j})) = \det (-F, H^*_\textrm{\'et,c}(\GG_m \otimes \bar{\FF}_q, \mathcal{M}^{i,j}))^{-1}\]
where $F$ denotes the geometric Frobenius relative to $\FF_q$ on $\GG_m \otimes \bar{\FF}_q$, see Theorem~VI.13.3 in \cite{Mi}. The $\ell$-adic cohomology group $H^k_\textrm{\'et,c}$ vanishes if $k \not= 1$ and has dimension  one if $k=1$ by the Grothendieck-Ogg-Shafarevic formula.
Hence the determinant of $-F$ is just the trace of $-F$ on $H^1_\textrm{\'et,c}$ which, by the Grothendieck trace formula 
is exactly the Gauss sum corresponding to the character~$\omega^i$. Formula~(\ref{ExamplePAdicValuation}) now follows from Stickelberger's formula as in the proof above. In the cases $i\not= 0, j=0$ and $i=0, j\not=0$, formula~(\ref{ExamplePAdicValuation}) follows from the fact that the $\ell$-adic cohomology groups $H^k_\textrm{\'et,c}(\GG_m \otimes \bar{\FF}_q, \mathcal{K}^{\otimes i})$ and $H^k_\textrm{\'et,c}(\AA^1_{\FF_q} \otimes \bar{\FF}_q, \mathcal{L}^{\otimes j})$, respectively, vanish for {\em all} $k$, again by the Grothendieck-Ogg-Shafarevic formula. Finally, the case $i=0, j=0$ is trivial.\\

{\bf Remark A.5}. Proposition~A.3 also shows that the class $j_p^{-1}(E(G,X)) \in K_0(\bar{\QQ}[G])_\QQ$ depends on $j_p$ (if $q > 3$); in other words, $E(G,X)$ does not descend to a canonical epsilon representation inside $K_0(\bar{\QQ}[G])_\QQ$. Indeed, if $j'_p:\bar{\QQ} \hookrightarrow \bar{\QQ}_p$ is another embedding, then $j'_p j_p^{-1}$ sends $[\omega]$ to $[\omega^{i'}]$ for some $i'$ coprime to $q-1$. If $q > 3$, it is easy to construct a $j'_p$ such that $s_p(i') \not= s_p(1) =1$ and hence that $j_p^{-1}(E(G,X)) \not= (j'_p)^{-1}(E(G,X))$.\\

\begin{center} {\bf Second Example} \end{center}

We finally construct a Galois cover $\pi: X \rightarrow \PP^1_{\FF_2}$ for which the denominator $2$ does occur in $E(G,X)$. Let $X$ be the elliptic curve given by the non-singular Weierstrass equation $v^2+v = u^3+u^2$ over $\FF_2$. We now define $\pi: X \rightarrow \PP^1_{\FF_2}$ to be the cover corresponding to the field extension $\FF_2(u,v)/\FF_2(t)$ where $t=u^2+u$. Thus, $\FF_2(u,v)/\FF_2(u)/\FF_2(t)$ is a tower of Artin-Schreier extensions. For the reader with the relevant background we mention that this tower can be seen to be associated with the Artin-Schreier-Witt extension corresponding to the equation $(u,v)^F - (u,v) = (t,0)$ of Witt vectors of length $2$ where $F$ denotes the Frobenius. (This second example has actually been found this way.) We see that $\pi$ is of degree four and that the Galois group~$G$ is generated by the automorphism~$\sigma$ of order four given by $\sigma(u) = u+1$ and $\sigma(v) = v + u +1$. Furthermore, $\pi$ is unramified above $\AA^1_{\FF_2} \subset \PP^1_{\FF_2}$ and totally ramified above $\infty \in \PP^1_{\FF_2}$.
For the unique point $\widetilde{\infty} \in X$ above $\infty$ we then have $\mathrm{ord}_{\widetilde{\infty}}(u) = -2$ and hence $\mathrm{ord}_{\widetilde{\infty}}(v) = -3$. Therefore, $\pi:= vu^{-2}$ is a uniformiser at $\widetilde{\infty}$ and we compute
\[\mathrm{ord}_{\widetilde{\infty}}(\sigma^2(\pi) - \pi) = \mathrm{ord}_{\widetilde{\infty}}((v+1)u^{-2} - vu^{-2}) = \mathrm{ord}(u^{-2}) = 4\]
and
\begin{align*}\mathrm{ord}_{\widetilde{\infty}}&(\sigma(\pi) - \pi) =  \mathrm{ord}_{\widetilde{\infty}}\left(\frac{v+u+1}{(u+1)^2} - \frac{v}{u^2}\right) \\
&=\mathrm{ord}_{\widetilde{\infty}}\left(\frac{u^3+u^2+v}{(u+1)^2 u^2}\right) = -6-(-4)-(-4) =2.
\end{align*}
For the ramification groups at $\widetilde{\infty}$ we therefore obtain
\[G_0 = G_1 = \langle \sigma \rangle, \quad G_2 = G_3 = \langle \sigma^2 \rangle, \quad G_4=G_5= \ldots = \langle \mathrm{id} \rangle\]
and hence
\[G^1= \langle \sigma \rangle, \quad G^2 = \langle \sigma^2 \rangle, \quad G^3 = G^4 = \dots = \langle \mathrm{id} \rangle.\]
Now let $\chi: G \hookrightarrow \bar{\QQ}^\times$ be a faithful character. Then $\mathrm{a}(\chi_{\widetilde{\infty}})=3$ and hence
$\varepsilon(\chi_{\widetilde{\infty}}, \psi_\infty, \mathrm{d}x_\infty)^2 \sim 2^3$
by Lemma~A.1; thus, the $2$-adic valuation $v_2(j_2(\varepsilon(\chi_{\widetilde{\infty}}, \psi_\infty, \mathrm{d}x_\infty)))$ is equal to $\frac{3}{2} \not\in \ZZ$. As in Theorem~A.2 and Proposition~A.3, this implies that $E(G,X)$ does not belong to $K_0(\bar{\QQ}_2[G])$.\\

{\bf Remark A.6}. Quicker, but less explicit and building on more theory, is the following approach to showing the existence of an example such that $E(G,X) \not\in K_0(\bar{\QQ}_2[G])$. Pick a local character $\rho_\infty: G_{\FF_2((t))} \rightarrow \bar{\QQ}^\times$ whose image is finite and whose Artin conductor is odd and bigger than one. (This can be done explicitly as above or using local class field theory.) Then the Katz-Gabber local-to-global extension \cite{Ka} gives a smooth $\ell$-adic sheaf $\mathcal{L}$ ($\ell \not= 2$) on $\GG_m/\FF_2$ which is tame at $0$ and has monodromy $\rho_\infty \otimes \QQ_l$ at~$\infty$. In particular, $\mathcal{L}$ has finite global monodromy. Now, as in Remark~A.4, we define~$\pi: X \rightarrow \PP^1_{\FF_2}$ as the smallest Galois cover which, over $\GG_m$, is \'etale and such that $\pi_{|{\GG_m}}^*(\mathcal{L})$ is trivial. Finally, by using the product formula as above, one shows that $E(G,X) \not\in K_0(\bar{\QQ}_2[G])$. (Note that, since $p=2$ and since $\mathcal{L}$ is tame at $0$, the $2$-adic valuation of the local epsilon constant at $0$ is integral.)

\Addresses
\end{document}